\documentclass[12pt]{amsart}
\usepackage{amssymb,latexsym,amsmath,amsthm,enumitem,mathrsfs,geometry}
\usepackage{fullpage}
\usepackage{fancyhdr}
\usepackage{xcolor}
\usepackage{bbm}
\usepackage{stmaryrd}
\usepackage{mathrsfs}
\usepackage{hyperref}
\usepackage{lineno}
\usepackage{diagbox}
\usepackage{array}
\usepackage{tikz}
\usetikzlibrary{arrows}
\usetikzlibrary{positioning}
\usepackage{float}
\usepackage{tabularx} 
\usepackage{array}     {\centering\arraybackslash}
\usepackage{comment}
\usepackage{mathtools}
\usepackage{cleveref}
\usepackage{tikz}
\usepackage{tikz-3dplot}
\usepackage{listings}
\usepackage{xcolor}
\usepackage{textcomp}
\usepackage[T1]{fontenc}
\usepackage{lmodern}
\usepackage{amsmath}   
\usepackage{listings}   
\definecolor{mygreen}{rgb}{0,0.6,0}

\usepackage[T1]{fontenc}
\usepackage{amsmath}
\usepackage{xcolor}
\usepackage{listings}
\usepackage{geometry}
\definecolor{codegreen}{rgb}{0,0.6,0}
\definecolor{codegray}{rgb}{0.5,0.5,0.5}
\definecolor{codepurple}{rgb}{0.58,0,0.82}
\definecolor{backcolour}{rgb}{0.95,0.95,0.95}
\definecolor{keywordblue}{rgb}{0.13, 0.13, 1}

\lstdefinestyle{mystyle}{
	backgroundcolor=\color{backcolour},   
	commentstyle=\color{codegreen},
	keywordstyle=\color{keywordblue}\bfseries,
	numberstyle=\tiny\color{codegray},
	stringstyle=\color{codepurple},
	basicstyle=\ttfamily\small,
	breakatwhitespace=false,         
	breaklines=true,                 
	captionpos=b,                    
	keepspaces=true,                 
	numbers=left,                    
	numbersep=5pt,                  
	showspaces=false,                
	showstringspaces=false,
	showtabs=false,                  
	tabsize=2,
	language=Mathematica,
	frame=single,
	rulecolor=\color{black!30},
	mathescape=true, 
	morekeywords={Reduce, Reals} 
}
\usepackage{tikz}
\usepackage{pgfplots}
\pgfplotsset{compat=1.18}
\usepackage{tikz-3dplot}
\usepackage{caption}
\usepackage{amsmath}
\usetikzlibrary{arrows.meta}

\newtheorem{theorem}{Theorem}[section]
\newtheorem*{theorem*}{Theorem}
\newtheorem{lemma}{Lemma}[section]

\newtheorem*{remark*}{Remark}

\newtheorem{proposition}[theorem]{Proposition}

\newcommand\JP[1]{\langle#1\rangle}
\newcommand\norm[1]{\left\|#1\right\|}

\def\lkr#1{\langle #1\rangle_{\omega}}

\newcommand\RE{\text{Re}}
\newcommand\IM{\text{Im}}

\newcommand{\R}{\mathbb{R}}
\newcommand{\N}{\mathbb{N}}
\newcommand{\C}{\mathbb{C}}
\newcommand{\Z}{\mathbb{Z}}
\def\d{\mathrm{d}}
\def\T{\mathbb{T}}

\def\cK{{\mathcal{K}}}
\def\cS{{\mathcal{S}}}
\def\cC{{\mathcal{C}}}
\def\cJ{{\mathcal{J}}}
\def\cL{{\mathcal{L}}}

\begin{document}
	
	\numberwithin{equation}{section}
	
	\title{Dispersive estimates for discrete Klein-Gordon equations on the one-dimensional lattice with quasi-periodic potentials}
	
	\author[Wan]{Zhiqiang Wan}
	\address{School of Mathematical Sciences, University of Science and Technology of China, Hefei}
	\email{ZhiQiang\_Wan576@mail.ustc.edu.cn}
	\author[Zhang]{Heng Zhang}
	\address{School of Mathematical Sciences, University of Science and Technology of China, Hefei}
	\email{hengz@mail.ustc.edu.cn}

	\subjclass[2020]{}
	\thanks{}

	\date{\today}

\begin{abstract}
    We prove $\ell^{1}\to\ell^{\infty}$ dispersive estimates for the discrete Klein--Gordon equations on $\mathbb Z$ with small real-analytic quasi-periodic potentials, showing that the time-decay rate persists as $(\tfrac13)^{-}$. As applications, we derive the corresponding Strichartz estimates. Further, we establish the global well-posedness and a scattering theorem for the associated nonlinear discrete Klein--Gordon equations with small data.
\end{abstract}
	\maketitle	
	
	\section{Introduction}

	We consider the Klein--Gordon equation with quasi-periodic potentials on the one-dimensional lattice $\Z$:
	\begin{equation}\label{eq:DKG}
		\begin{cases}
			\bigl(\partial_{tt}+G_{\theta}+m^{2}\bigr)u(n,t)=0,\qquad\qquad\quad(n,t)\in\mathbb{Z}\times\left(0,\infty\right),\\
            u(n,0)=\varphi(n)\in\ell^{1}\left(\Z\right),\quad \partial_{t}u(n,0)=\psi(n)\in\ell^{1}\left(\Z\right),\quad n\in\mathbb{Z},\\
		\end{cases}
	\end{equation}
	where $G_{\theta}$ is defined as
	\begin{equation*}
		(G_{\theta}u)(n)
		:=-\Delta_{\text{disc}}u(n)+V(\theta+n\omega)u(n),
		\qquad n\in\mathbb{Z},
	\end{equation*}
	and the discrete Laplacian is given by
	$\Delta_{\text{disc}}u(n):=u(n+1)+u(n-1)-2u(n)$.
	Here $m>0$, $\theta \in \T^d:=\left(\R/2\pi\Z\right)^{d}$, the potential $V\colon\mathbb{T}^{d}\to\mathbb{R}$ is real-analytic with $d\geq 1$, and the frequency vector $\omega\in\mathbb{R}^{d}$ is \textbf{Diophantine}; i.e., there exist constants $\gamma>0$ and $\eta>d-1$ such that
    \begin{equation}\label{dio}
		\inf_{j\in\mathbb{Z}}\bigl|\langle k,\omega\rangle_{\R^{d}}-2j\pi\bigr|
		>\frac{\gamma}{|k|^{\eta}},
		\qquad \forall\,k\in\mathbb{Z}^{d}\setminus\{0\},
    \end{equation}
    where $\langle\cdot{,}\cdot\rangle_{\R^{d}}$ denotes the standard scalar product on $\mathbb{R}^{d}$. We assume that $V$ admits a bounded analytic complex extension to $\left\{z\in\C^{d}:\left|{\rm Im}z\right|<r\right\}$. We also denote
    \begin{equation*}
        \varepsilon_{0}:=\left|V\right|_{r}=\sup_{\left\{z\in\C^{d}:\left|{\rm Im}z\right|<r\right\}}\left|V\left(z\right)\right|.
    \end{equation*}
    In the literature, one often sets $H_{\theta}=G_{\theta}-2$ to be the \textbf{quasi-periodic Schr\"odinger operator}.  We are interested in the dispersive estimates for the equation \eqref{eq:DKG}, namely the $\ell^{1}\!\to\!\ell^{\infty}$ estimates of the form
	\begin{equation*}
		\norm{u(\cdot,t)}_{\ell^{\infty}(\Z)}
		\leq C\JP{t}^{-\tau}\bigl(\norm{\varphi}_{\ell^{1}(\Z)}+\norm{\psi}_{\ell^{1}(\Z)}\bigr),
		\qquad \forall\,\varphi,\psi\in\ell^{1}(\Z),
	\end{equation*}
	where $\JP{t}:=\sqrt{1+t^{2}}$ is the Japanese bracket, $C>0$ is a constant independent of $t$ and $\tau$ is the so-called dispersive decay rate.

	\textbf{Motivation} The goal of this work is to understand to what extent dispersive properties of lattice Klein–Gordon dynamics are stable under introducing a quasi-periodic medium. On the one hand, the free discrete Klein–Gordon equation on $\mathbb Z$ is known to exhibit a precise dispersive decay rate (see \cite{SK05}), which plays a central role in both the linear and nonlinear theory. On the other hand, quasi-periodic Schr\"odinger operators $H_\theta$ are canonical models for wave propagation in quasicrystals and are known, in the small analytic and Diophantine regime, to have purely absolutely continuous spectrum of Cantor type (see for example \cite{AD08, Avi08}). From a dispersive PDE viewpoint, this spectral picture is highly singular and lies completely outside the standard setting of the fine continuous spectrum given by a finite union of intervals, so it is a priori unclear whether any robust $\ell^1\to\ell^\infty$ decay should survive. 
    
    Our main result states as follows.
\begin{theorem}\label{thm:disper}
		Let $u$ be a solution of equation \eqref{eq:DKG}. Assume $\varepsilon_0<\varepsilon_{**}$, with $\varepsilon_{**}$ small as in Lemma \ref{lem:disper-log}. Then for any given $0<\tau<\frac13$, there exists $K_1=K_1(\varepsilon_0,\eta,m,\tau)$ such that for any $\theta\in\T^d$ and any $t\in\R$,
\begin{equation}\label{linest}
			\norm{u\left( \cdot ,t\right)}_{\ell^\infty\left(\Z\right)}\leq K_1\JP{t}^{-\tau}(\norm{\varphi}_{\ell^1\left(\Z\right)}+\norm{\psi}_{\ell^1\left(\Z\right)}) \ , \quad \forall \ \varphi, \psi \in\ell^1(\Z) \ .
		\end{equation}
	\end{theorem}
	For our Theorem \ref{thm:disper}, we have the following remarks.
	
	\begin{itemize}
		\item[($\mathbf{R}_1$)] Stefanov and Kevrekidis~\cite{SK05} first established a sharp dispersive decay rate  
		$
		\frac{1}{3}
		$ 
		for the free discrete Klein-Gordon equation (i.e., equation \eqref{eq:DKG} with $V=0$) on $\mathbb Z$ via Van der Corput's lemma.  
		Recently, Borovyk and Goldberg \cite{BG17} derived the corresponding rate on $\mathbb Z^{2}$, namely  $\langle t\rangle^{-\frac{3}{4}}$. Later, Cuenin and Ikromov~\cite{CI21} derived the rates on~$\mathbb Z^{d}$ for $d=3,4$, namely $
		\langle t\rangle^{-\frac{7}{6}}\ (d=3)$, and
		$\langle t\rangle^{-\frac{3}{2}}\log \bigl(2+|t|\bigr)\ (d=4)$. The obstacle to extending Cuenin-Ikromov to higher dimensions is that the finite classification of stable phase singularities (Thom's catastrophes), which is essential for deriving sharp decay rates, breaks down in dimensions $d \ge 5$. Our result can be seen as a perturbation result for Stefanov and Kevrekidis, which shows that the decay rate $(\frac{1}{3})^{-}$ persists under small quasi-periodic perturbations of the potentials.

		\item[($\mathbf{R}_2$)] While several studies (e.g., \cite{CT09, GPP13, PS08, KPS09, KKK06}) have addressed dispersive equations with potentials whose continuous spectrum is a union of disjoint intervals, generic behavior of the operator $H_{\theta}$ is markedly different: its spectrum is purely absolutely continuous and forms a Cantor set; see, e.g., \cite{AD08, Avi08, E92}. Our approach builds heavily on recent advances in the spectral theory of quasi-periodic Schr\"odinger operators by Zhao \cite{Z16}. Specifically, we employ a KAM-based perturbative construction to approximate the rotation number and decompose the spectrum into many small intervals. Using a modified  spectral transform, we represent the solution as an oscillatory integral and then apply Van der Corput's lemma on each interval to establish dispersive decay. We highlight that a closely related strategy is employed in Zhao’s work on ballistic transport \cite{Z16} and in the dispersive estimates of Bambusi–Zhao for discrete Schrödinger equations with quasi-periodic potentials \cite{BZ20}.

        \item[($\mathbf{R}_3$)] We compare our result with the recent work of Cheng~\cite{C26}. 
The first version of the present paper was posted on arXiv before Cheng's paper became publicly available as an arXiv preprint; to the best of our knowledge, no arXiv version of~\cite{C26} was available at the time when the first arXiv version of the present paper was posted. We have added the reference~\cite{C26} in the revised manuscript for completeness. Let us now explain the mathematical relation between the two works. In the notation of the present paper, \cite{C26} corresponds to the special mass case $m^2=1$. The passage from the case \(m^2=1\) to arbitrary \(m^2>0\) is not a formal rescaling in the lattice setting. 
Indeed, if one rescales time by \(s=mt\), then
\[
        \bigl(\partial_t^2-\Delta_{\rm disc}+m^2+V(\theta+n\omega)\bigr)u_n=0
\]
is transformed into
\[
        \partial_s^2u_n
        +m^{-2}\bigl(-\Delta_{\rm disc}+V(\theta+n\omega)\bigr)u_n
        +u_n=0 .
\]
The coefficient \(m^{-2}\) in front of the discrete Laplacian cannot be removed by a spatial dilation, since the underlying space is the fixed lattice \(\mathbb Z\), and a dilation of the lattice variable does not preserve either \(\mathbb Z\), the nearest-neighbor Laplacian, or the quasi-periodic sampling \(\theta+n\omega\). This is a genuine difference from the Euclidean continuum case. Consequently, the \(m=1\) estimate in~\cite{C26} does not imply the dispersive estimate for general $m^2>0$. 
        
\item[($\mathbf{R}_4$)] For the free discrete wave equation with quasi-periodic potentials (i.e., equation \eqref{eq:DKG} with $m=0$), there do not exist $\ell^1 \to \ell^{\infty}$ estimates for the fundamental solution of the free wave equation. This phenomenon also holds on $\R^1$. We point out the following development regarding the dispersive estimate for the free discrete wave equation.  Schultz \cite{Sch98} proved dispersive estimates for the wave equation on lattice graphs $\mathbb{Z}^d$ for $d=2,3$, namely  $\langle t\rangle^{-\frac{3}{4}}\ (d=2)$ and $\langle t\rangle^{-\frac{7}{6}}\ (d=3)$. Combining the singularity theory
		with results in uniform estimates of oscillatory integrals, Bi-Chen-Hua \cite{BCH1} extended this result  to $d=4$: they proved that the optimal
		time decay rate is $|t|^{-\frac{3}{2}} \log |t|$. By tools from Newton polyhedra, Bi-Chen-Hua \cite{BCH2} further extended the result to $d=5$: the sharp decay rate of the fundamental solution of the wave equation on $\mathbb{Z}^5$ is $|t|^{-\frac{6}{11}}$.
    \end{itemize}
	
	A typical application for dispersive estimates is the derivation of Strichartz inequalities. For the inhomogeneous discrete Klein-Gordon equation, 
	\begin{equation}\label{eq: inhomogegeneous}
		\begin{cases}
			\bigl(\partial_{tt}+G_{\theta}+m^{2}\bigr)u(n,t)=F(n,t),
			&(n,t)\in\mathbb{Z}\times\left(0,\infty\right),\\[2mm]
			u(n,0)=\varphi(n),\quad \partial_{t}u(n,0)=\psi(n),
			&n\in\mathbb{Z},
		\end{cases}
	\end{equation}
	we have the following mixed space-time estimates via the standard Keel-Tao argument. For convenience, we say   
	a pair of exponents $(q,r)$ with
	$
	2\le q,r\le\infty
	$
	is called \emph{$\tau$--admissible} if
	$
	\frac{2}{q}
	=\tau\Bigl(1-\frac{2}{r}\Bigr).
	$ Here $0<\tau<\tfrac13$ is fixed as in Theorem~\ref{thm:disper}.
	
	\begin{theorem}\label{thm:Strichartz}
		 Assume the hypotheses of Theorem~\ref{thm:disper}. Fix
		$0<\tau<\tfrac13$ and $F\in L^{\tilde q'}(\mathbb R;\ell^{\tilde r'}(\mathbb Z))$. Let $(q,r)$ and $(\tilde q,\tilde r)$ be two $\tau$--admissible pairs with $q,\tilde q<\infty.$  Then the solution 
		$
		u
		$ of \eqref{eq: inhomogegeneous}
satisfies\begin{equation}\label{eq:inhom-Strichartz}
			\|u\|_{L^q_t(\mathbb R;\ell^r(\mathbb Z))}
			\;\le\; C\Bigl(
			\|\varphi\|_{\ell^2(\mathbb Z)}
			+\|\psi\|_{\ell^2(\mathbb Z)}
			+\|F\|_{L^{\tilde q'}_t(\mathbb R;\ell^{\tilde r'}(\mathbb Z))}
			\Bigr),
		\end{equation}
		where $\tilde q'$ and $\tilde r'$ denote the Hölder conjugates of
		$\tilde q$ and $\tilde r$, respectively, and the constant $C$ is
		independent of $\theta$, $\varphi$, $\psi$ and $F$.

	\end{theorem}
	
	Using Strichartz estimates, we establish the global well-posedness of the nonlinear Klein-Gordon equation  \begin{equation}\label{eq: nonlinear}
		\begin{cases}
			\bigl(\partial_{tt}+G_{\theta}+m^{2}\bigr)u(n,t)= \pm |u|^{p-1}u,
			&(n,t)\in\mathbb{Z}\times\left(0,\infty\right),\\[2mm]
			u(n,0)=\varphi(n),\quad \partial_{t}u(n,0)=\psi(n),
			&n\in\mathbb{Z},
		\end{cases}
	\end{equation}
with small initial data.	
	\begin{theorem}\label{thm:nonlinear}
		Let $p>7$. There exist $\varepsilon > 0$ and a constant $C$, so that whenever
		$\| \varphi \|_{\ell^2\left(\Z\right)}+\| \psi\|_{\ell^2\left(\Z\right)} \leq \varepsilon
		$, there exists a unique global solution to \eqref{eq: nonlinear}, satisfying
		\begin{equation*}
		\|u\|_{L^q_t(\mathbb R;\ell^r(\mathbb Z))}\leq C\varepsilon,
		\end{equation*}
		for all $\tau$-admissible pairs $(q, r)$. In particular, $\| u(n,t) \|_{\ell^r(\Z)} \to 0$ as $t \to \infty$ for every $r > 2$.  
	\end{theorem}

Further, we establish the small-data scattering theorem for equation \eqref{eq: nonlinear}. 

\begin{theorem}\label{thm:scatter}
Let $p>7$. There exists $\delta>0$ such that for any
$(\varphi_+,\psi_+)\in \ell^2(\Z)\times \ell^2(\Z)$ with
$\|\varphi_+\|_{\ell^2}+\|\psi_+\|_{\ell^2}\le \delta$,
there  exists a unique datum $(\varphi,\psi)\in \ell^2(\Z)\times \ell^2(\Z)$
so that the global solution $u$ to \eqref{eq: nonlinear} satisfies the following:
there exists a (unique) solution $v_+$ to the linear equation
\[
(\partial_{tt}+G_\theta+m^2)v_+=0,\qquad v_+(0)=\varphi_+,\quad \partial_t v_+(0)=\psi_+,
\]
such that
\[
\lim_{t\to +\infty}
\left\|
\binom{u(t)}{\partial_t u(t)}-\binom{v_+(t)}{\partial_t v_+(t)}
\right\|_{\ell^2(\Z)\times \ell^2(\Z)}=0.
\]
An analogous statement holds as $t\to-\infty$.
\end{theorem}

The remainder of the paper is organized as follows. Section \ref{sec:2} reviews the structure of the spectrum of the quasi-periodic Schr\"odinger operator and presents a spectral transform that will be used throughout. Section \ref{sec:3} establishes Theorem \ref{thm:disper} by analysing a specific oscillatory integral on the spectrum. In Section \ref{sec:4}, we prove Theorems \ref{thm:Strichartz}, \ref{thm:nonlinear} and \ref{thm:scatter}.
	
\section{Preliminaries}\label{sec:2}

	\subsection{Schr\"odinger cocycle and fibered rotation number}

 The spectral problem $H_{\theta}u=Eu$ is equivalent to the following first-order system, known as the Schr\"odinger cocycle $(\omega, A_0+F_0)$:
	
	\begin{equation}\label{eq:cocycle}
		\begin{pmatrix}
			u(n+1) \\
			u(n)\\
		\end{pmatrix}
		=(A_0(E)+F_0(\theta+n\omega))
		\begin{pmatrix}
			u(n) \\
			u(n-1) \\
		\end{pmatrix}
		\ ,
	\end{equation}
	with $A_0(E):=
	\begin{pmatrix}
		-E&-1\\
		1&0\\
	\end{pmatrix}$ and $F_0(\theta):=
	\begin{pmatrix}
		V(\theta) & 0 \\
		0 & 0
	\end{pmatrix}$. 
	
	We begin with recalling the definition of the fibered rotation number,
	originally introduced by Herman~\cite{Her83}. The following exposition follows
	\cite{HA09}.
	
	Let
	\[
	A + F : \theta \in \T^d \longmapsto
	\begin{pmatrix}
		a(\theta) & b(\theta) \\
		c(\theta) & d(\theta)
	\end{pmatrix}
	\in SL(2,\R),
	\]
	and consider the induced map
	\[
	T_{(\omega, A+F)} : (\theta, \varphi) \in \T^d \times \tfrac{1}{2}\T
	\longmapsto
	(\theta + \omega,\, \phi_{(\omega, A+F)}(\theta, \varphi))
	\in \T^d \times \tfrac{1}{2}\T,
	\]
	where
	\[
	\phi_{(\omega, A+F)}(\theta, \varphi)
	= \arctan
	\!\left(
	\frac{c(\theta)+d(\theta)\tan\varphi}{a(\theta)+b(\theta)\tan\varphi}
	\right).
	\]
	Assume that \(a,b,c,d\) are continuous on \(\T^d\) and that the cocycle \(A+F(\theta)\)
	is homotopic to the identity. Then the same holds for \(T_{(\omega, A+F)}\),
	which consequently admits a continuous lift
	\[
	\widetilde{T}_{(\omega, A+F)} : (\theta, \varphi) \in \T^d \times \R
	\longmapsto
	(\theta + \omega,\, \tilde{\phi}_{(\omega, A+F)}(\theta, \varphi))
	\in \T^d \times \R
	\]
	such that
	\[
	\tilde{\phi}_{(\omega, A+F)}(\theta, \varphi) \bmod \pi
	= \phi_{(\omega, A+F)}(\theta, \varphi \bmod \pi).
	\]
	The function
	\[
	(\theta, \varphi) \longmapsto
	\tilde{\phi}_{(\omega, A+F)}(\theta, \varphi) - \varphi
	\]
	is \((2\pi)^d\)-periodic in \(\theta\) and \(\pi\)-periodic in \(\varphi\). Set
	\[
	\rho(\phi_{(\omega, A+F)})
	:=
	\lim_{n \to +\infty}
	\frac{1}{n}
	\bigl(p_2 \circ \widetilde{T}^{\,n}_{(\omega, A+F)}(\theta, \varphi) - \varphi\bigr)
	\in \R,
	\]
	where \(p_2(\theta, \varphi) = \varphi\).
	This limit exists for all \((\theta, \varphi) \in \T^d \times \R\),
	and the convergence is uniform in both variables
	(see~\cite{Her83} for details). The equivalence class of \(\rho(\phi_{(\omega, A+F)})\) in \(\tfrac{1}{2}\T\),
	which is independent of the chosen lift, is called the
	\textbf{fibered rotation number} of the skew-product system
	\[
	(\omega, A+F) : (\theta, y) \in \T^d \times \R^2
	\longmapsto
	(\theta + \omega,\; (A + F(\theta))y)
	\in \T^d \times \R^2,
	\]
	and is denoted by \(\rho_{(\omega, A+F)}\).

	\subsection{Structure of the spectrum}
	We review here the KAM theory of Eliasson \cite{E92} and Hadj Amor
	\cite{HA09} for the reducibility of the Schr\"odinger cocycle $(\omega,
	A_0+F_0(\cdot))$. These works relate the reducibility and the fibered
	rotation number globally, and improve the previous works by Dinaburg-Sinai
	\cite{DS75} and Moser-P\"oschel \cite{MP84}. Here we will not prove the
		corresponding results (Propositions \ref{propsana} and \ref{propsana1})
		referring to the work \cite{Z16} (see also \cite{BZ20}) where a detailed proof was
		given.
        
Throughout this paper, we set
\[
\varepsilon_0 = |V|_r, \quad \sigma = \frac{1}{200},
\]
and define the sequences as in \cite{HA09}:
\[
\varepsilon_{j+1} = \varepsilon_j^{1+\sigma}, \quad N_j = 4^{j+1} \sigma |\ln \varepsilon_j|, \quad j \geq 0.
\]
We also define
\begin{equation}\label{stor}
\lkr k := \frac{\langle k, \omega \rangle}{2}, \quad k \in \mathbb{Z}^d,
\end{equation}
and denote by \(|\cdot|_{\mathcal{C}_W^k(S)}\) the \(\mathcal{C}^k\) norm of a function that is Whitney smooth on a set \(S \subset \mathbb{R}\). For a function that is analytic on \(\mathbb{T}^d\) (or \(2\mathbb{T}^d\)) and Whitney smooth on \(S\), we denote by
\[
|\cdot|_{\mathcal{C}_W^k(S), \mathbb{T}^d}
\quad \text{or} \quad
|\cdot|_{\mathcal{C}_W^k(S), 2\mathbb{T}^d}
\]
the supremum norm on \(\mathbb{T}^d\) (or \(2\mathbb{T}^d\)) and the \(\mathcal{C}_W^k\) norm on \(S\). In particular, if \(S\) is a union of finitely many intervals, we omit the subscript \(W\) in these norms.
	
	Furthermore, we denote the fibered rotation number of the Schr\"odinger cocycle
	$(\omega,A_0+F_0)$ by $\rho\equiv\rho_{(\omega,A_0+F_0)}$. Set $\Sigma$ to be the spectrum of $H_{\theta}$.
	We mention that $\Sigma$ is independent of $\theta$ and $\rho:\R\to [0,\pi]$ is a non-decreasing function with
	$$\rho(E)\left\{ \begin{array}{ll}
		=0 \ , & E\leq \inf\Sigma  \\
		\in (0,\pi) \ ,& E\in (\inf\Sigma, \sup\Sigma) \\
		=\pi \ ,& E\geq \sup\Sigma
	\end{array}  \right.  \ , $$
	By the gap-labeling theorem \cite{JM82}, $\rho$ is constant in a gap of $\Sigma$  (i.e., an interval on $\R$ in the resolvent set of $H_\theta$), and each gap is labeled with
	$k\in \Z^d$ such that $\rho=\lkr k$ mod $\pi$ in this gap.

	The following proposition constitutes the foundational result of the KAM (Kolmogorov-Arnold-Moser) iteration scheme in the infinite limit. It establishes that for sufficiently small potentials, the Schr\"odinger cocycle is reducible on a full-measure subset of the spectrum. By characterizing the spectrum $\Sigma$ as a Cantor set and providing the existence of the limit conjugacy matrices, this theorem defines the global landscape on which the dynamics occur. Its primary role in the subsequent analysis is to rigorously control the measure of the resonant sets excluded during the iteration. These measure estimates are indispensable for the approximation arguments in Lemma \ref{lem:disper-log}, ensuring that the integral over the fractal Cantor set $\Sigma$ is well-defined and can be approximated by integrals over finite unions of intervals with controllable error.
	
	\begin{proposition}[Proposition 1 of \cite{Z16}]\label{propsana} There exists $\varepsilon_*=\varepsilon_*(\gamma,\eta, r,d)>0$ such that if $|V|_r=\varepsilon_0<\varepsilon_*$,
		then, for any $j\in \N$, there exists a Borel set $\Sigma_{j}\subset
		\Sigma$, with $\{\Sigma_{j}\}_j$ mutually disjoint, satisfying
		\begin{align*}
			|\rho\left(\Sigma_{j+1}\right)|&\leq 3 |\ln\varepsilon_j|^{2 d}
			\varepsilon_{j}^{\sigma} \ , \quad j\geq 0 \ ,
			\\ \left|\Sigma\setminus\widetilde \Sigma\right|&=0 \ ,\quad \widetilde
			\Sigma:=\cup_{j\geq 0}\Sigma_{j},
		\end{align*}
		such that the following statements hold.
		\begin{itemize}
			\item [(1)] The Schr\"odinger cocycle $(\omega, A_0+F_0)$ is {\it
				reducible} on $\widetilde
			\Sigma$. More precisely, there
			exist $Z$ and $B$, with $Z:\widetilde
			\Sigma \times2\T^d\to SL(2,\R)$
			analytic on $2\T^d$ and $B:\widetilde
			\Sigma\to SL(2,\R)$
			s.t. $Z$ conjugates $A_0+F_0$ to $B$, namely
			$$Z(\cdot+\omega)^{-1} (A_0+F_0(\cdot)) \, Z(\cdot)=B \ .$$
			Furthermore
			$B$ is ${\cC}^1$ in the sense of Whitney on each ${\Sigma}_j$, and
			\begin{equation}\label{limit_state_whitney}
				|B-A_0|_{{\cC}^1_W({\Sigma}_0)}\leq \varepsilon_0^{\frac13} \ ; \qquad |B|_{{\cC}^1_W({\Sigma}_{j+1})}\leq N_j^{10\eta} \ ,\quad j\geq 0 \ .
			\end{equation}
			\item [(2)] The eigenvalues of $B\big|_{\Sigma_j}$, are of the form
			$e^{\pm{\rm i}\xi}$, with $\xi\in \R$, and, for every $j\geq 0$, there is $k_j:\widetilde
			\Sigma\rightarrow\Z^d$, such that
			\begin{itemize}
				\item[$\bullet$]$0<|k_j|\leq N_{j}$ on $\Sigma_{j+1}$, and $k_l=0$
				on $\Sigma_j$ for $l\geq j$,
				\item[$\bullet$]  $\xi=\rho-\sum_{l\geq0} \lkr{k_l}$ and $0<|\xi|_{\Sigma_{j+1}}< 2 \varepsilon_{j}^{\sigma}$.
			\end{itemize}
		\end{itemize}
	\end{proposition}

	While Proposition \ref{propsana} describes the asymptotic limit, Proposition \ref{propsana1} characterizes the system after a finite number $J$ of iterative steps. This result is the computational engine of the proof. Since the spectrum $\Sigma$ is fractal and totally disconnected, direct calculation is impossible. Proposition \ref{propsana1} circumvents this by providing a sequence of smooth approximate rotation numbers $\rho_J$ defined on a union of open intervals $\Gamma^{(J)}_j$. Crucially, it establishes quantitative lower bounds on the derivatives of the energy with respect to the rotation number. These derivative bounds are the specific inputs required to apply Van der Corput's lemma in Section \ref{sec:3}, which ultimately yields the decay rates of the oscillatory integrals.
	\begin{proposition}[Proposition 2 of \cite{Z16}]\label{propsana1}
		Let $|V|_r=\varepsilon_0< \varepsilon_*$ be as in Proposition \ref{propsana}.
		Given any $J\in\N$, for $0\leq j\leq J$, there exists $\Gamma^{(J)}_j\subset[\inf\Sigma, \sup\Sigma]$, satisfying
		\begin{itemize}
			\item $\Sigma_{j}\subset \Gamma^{(J)}_{j}$ for $0\leq j\leq J$,
			\item $\{\Gamma^{(J)}_j\}_{j=0}^{J}$ are mutually disjoint and $\overline{\bigcup_{j=0}^{J}\Gamma^{(J)}_j}=[\inf\Sigma,\sup\Sigma]$
			\item $\bigcup_{j=0}^J \Gamma^{(J)}_j$ consists of at most
			$|\ln\varepsilon_0|^{2 J^2 d}$ open intervals,
			\item If $J\geq 1$, then
			$
			\left|\rho\left(\Gamma^{(J)}_{j+1}\right)\right|\leq 3|\ln\varepsilon_j|^{2 d} \varepsilon_{j}^{\sigma}
			$ for $0\leq j \leq J-1$.
		\end{itemize}
		Furthermore, the following statements hold.
		
		\noindent
		{\bf (S1)}
		There exist
		$\left\{ \begin{array}{l}
			A_J:\Gamma_j^{(J)}\rightarrow SL(2,\R)\\[1mm]
			F_J:\Gamma_j^{(J)}\times \T^d\rightarrow gl(2,\R) \ analytic \  on \  \T^d\\[1mm]
			Z_J:\Gamma_j^{(J)}\times2\T^d\rightarrow SL(2,\R) \ analytic \  on \ 2\T^d
		\end{array}
		\right.$, $0\leq j \leq J$,
		
		\noindent all of which are smooth on each connected component of $\Gamma^{(J)}_j$,
		such that
		$$Z_J(\cdot+\omega)^{-1} (A_0+F_0(\cdot)) \, Z_J(\cdot)=A_J+F_J(\cdot) \ ,$$
		with {$|F_J|_{{\cC}^3(\Gamma^{(J)}_{j}),\T^d}\leq \varepsilon_J$}, $0\leq j \leq J$, and
		\begin{equation}\label{sigma_m_0}
			|A_{J}- A_0|_{{\cC}^3(\Gamma^{(J)}_{0})}\leq \varepsilon_0^{\frac12} \ , \qquad
			|Z_{J}-Id.|_{{\cC}^3(\Gamma^{(J)}_{0}),2\T^d}\leq \varepsilon_0^{\frac13} \ .
		\end{equation}
		If $J\geq 1$, then for $0\leq j\leq J-1$,
		\begin{equation}\label{sigma_m_J}
			|A_{J}|_{{\cC}^3(\Gamma^{(J)}_{j+1})}\leq \varepsilon_j^{-\frac{\sigma}6} \ , \qquad
			|Z_{J}|_{{\cC}^3(\Gamma^{(J)}_{j+1}),2\T^d}\leq \varepsilon_j^{-\frac{\sigma}3} \ ,
		\end{equation}
		and, on $\Gamma^{(J)}_{j+1}$,
		\begin{equation}\label{trace_A_J}
			\varepsilon_{j}^{\frac{\sigma}{4}}\leq |({\rm tr}A_J)'| \leq N_{j}^{10\eta} \ .
		\end{equation}
		Moreover, for $0\leq j\leq J$,
		\begin{equation}\label{error_whitney}
			|A_{J}- B|_{{\cC}^1_W(\Sigma_j)} \leq \varepsilon_{J}^{\frac14} \ ,\quad |Z_J-Z|_{{\cC}^1_W(\Sigma_j), 2\T^d}\leq \varepsilon_{J}^{\frac14} \ .
		\end{equation}
		
		\smallskip
		
		\noindent
		{\bf (S2)} $A_{J}$ has two eigenvalues $e^{\pm{\rm i}\alpha_{J}}$ with $\alpha_J\in \R \cup {\rm i}\R$. For $\xi_{J}:={\rm Re}\alpha_{J}$, we have
		\begin{itemize}
			\item $|\xi_{J}-\xi|_{\Sigma_j}\leq \varepsilon_{J}^{\frac14}$, $0\leq j\leq J$.
			\item $|\xi_{J}-\rho|_{\Gamma^{(J)}_{0}}\leq \varepsilon_J^{\frac14}$.
			\item If $J\geq 1$, then
			\begin{itemize}
				\item $|\xi_{J}|_{\Gamma^{(J)}_{j+1}}\leq \frac32\varepsilon_j^{\sigma}$, $0\leq j\leq J-1$.
				\item There is $k_j: \bigcup_{l=0}^{J}\Gamma^{(J)}_l\rightarrow\Z^d$, $0\leq j\leq J-1$, constant on each connected component of $\bigcup_{l=0}^{J}\Gamma^{(J)}_l$, with $0<|k_j|\leq N_{j}$ on $\Gamma^{(J)}_{j+1}$ and $k_l=0$ on $\Gamma^{(J)}_{j+1}$ for $l\geq j+1$ such that
				$\left|\xi_{J} + \sum_{l=0}^{J-1}\lkr{ k_l}-\rho\right|_{\Gamma^{(J)}_{j+1}}\leq \varepsilon_J^{\frac14}$.
			\end{itemize}
		\end{itemize}
		
		\smallskip
		
		\noindent
		{\bf (S3)} $\bigcup_{j=0}^J\{\Gamma^{(J)}_{j}: |\sin\xi_{J}|> \frac32 \varepsilon_J^{\frac{1}{20}}\}$ has at most $2|\ln\varepsilon_0|^{2J^2 d}$ connected components, on which $\xi_{J}$ is smooth with $\xi_{J}'=-\frac{({\rm tr}A_{J})'}{2\sin\xi_{J}}$. If $J\geq 1$, then, on $\{\Gamma^{(J)}_{j+1}: |\sin\xi_{J}|> \frac32 \varepsilon_J^{\frac{1}{20}}\}$, $0\leq j\leq J-1$,
		\begin{equation}\label{esti_plat}
			\frac13< \xi_{J}' \leq \frac{N_j^{10\eta}}{|\sin\xi_{J}|} \ , \qquad \frac{\varepsilon_{j}^{\frac{3\sigma}{4}}}{4|\sin\xi_{J}|^3}< |\xi_{J}'' |\leq \frac{N_j^{20\eta}}{|\sin\xi_{J}|^{3}} \ .
		\end{equation}
		
		\smallskip
		
		\noindent
		{\bf (S4)}
		$\left|\rho(\{(\inf\Sigma,\sup\Sigma): |\sin\xi_{J}\right|\leq \frac32 \varepsilon_J^{\frac{1}{20}} \})|\leq \varepsilon_J^{\frac{1}{24}}$ and for $0\leq j\leq J$, $|\xi_{J}(\Gamma^{(J)}_j\setminus\Sigma_j)|\leq \varepsilon_{J}^{\frac{7\sigma}{8}}$.

	\end{proposition}
	
	From now on, we denote $\rho_J:=\xi_{J} + \sum_{l=0}^{J-1}\lkr{ k_l}$,
	which gives an approximation of $\rho$. In particular, $\rho_0=\xi_0$,
	and
	\begin{equation}
		\label{rhoj}
		\left|\rho_{J}-\rho\right|_{\Sigma_j}\leq \varepsilon_j^{\frac14} \ .
	\end{equation}

	\subsection{Modified spectral transform}\label{spectrans}

	For every $E\in\Sigma$, let $\mathcal K(E)$ and $\mathcal J(E)$ be two linearly independent generalized eigenvectors of $H_\theta$.  
Define the spectral transform $\mathcal S$ as follows: for any $u\in\ell^2(\mathbb Z)$, set
\begin{equation}
  \label{spec}
  (\mathcal S u)(E):=
  \begin{pmatrix}
    \sum_{n\in\mathbb Z}u_n \mathcal K_n(E)\\[1mm]
    \sum_{n\in\mathbb Z}u_n \mathcal J_n(E)
  \end{pmatrix}
  :=
    \begin{pmatrix}
        g_{1}\left(E\right)\\
        g_{2}\left(E\right)\\
    \end{pmatrix}.
\end{equation}
	Given any matrix of measures on $\R$, namely $\d\varphi=\left(\begin{array}{cc}
		\d\varphi_{11} & \d\varphi_{12} \\[1mm]
		\d\varphi_{21} & \d\varphi_{22}
	\end{array}\right)$,
	let ${\cL}^2(\d\varphi)$ be the space of the vectors
	$G=(g_j)_{j=1,2}$, with $g_j$ functions of $E\in\R$ satisfying
	\begin{equation}\label{gen_L2}
		\|G\|_{{\cL}^2(\d\varphi)}^2:=\sum_{j,k=1}^2 \int_\R g_j \, \bar g_k \,  \d\varphi_{jk}<\infty \ .
	\end{equation}

	It is well known that (see e.g. \cite{CL55})  there exists a Hermitian matrix of measures $\mu=(\mu_{jk})_{j,k=1,2}$, with $\mu_{jk}$ non-decreasing functions, such that $\cS:\ell^2(\Z)\to \cL^2(\d\mu)$ is
		unitary. We remark that by this conclusion  the spectral transform is invertible. As anticipated in the introduction it is not known how to construct the measure $\d\mu$, however in \cite{Z16} a procedure to construct an approximate measure was developed.

	Serving as the bridge between the physical space and the spectral domain, Proposition \ref{prop:MST} addresses the lack of an explicit expression for the exact spectral measure $\d\mu$. Instead of relying on the abstract existence of a unitary transformation (as in Proposition \ref{prop:MST}), this proposition constructs an explicit, approximate spectral transform $\mathcal{S}$ using the generalized eigenfunctions derived from the reducibility matrices of Propositions \ref{propsana} and \ref{propsana1}. It utilizes the density $\rho^{\prime}\d E$ as a practical substitute for the exact measure. Although this transformation is not strictly unitary, it is shown to be a bounded operator with a bounded inverse (norm equivalence). This formulation allows the solution $u(t)$ to be represented as an oscillatory integral over the rotation number, transforming the problem of dispersive decay into the estimation of integrals handled by the machinery of the previous two propositions.
	\begin{proposition}[Section 4.2 in \cite{Z16}]\label{prop:MST}
		On the full measure subset
		$\widetilde{\Sigma}=\bigcup_{j\geq0}\Sigma_j$ of the spectrum, for any fixed $\theta\in \T^d$, any $E\in\widetilde{\Sigma}$, there exist two linearly independent
		generalized eigenvectors ${\cK}(E)$ and ${\cJ}(E)$ of $H_\theta$
		with the following properties: define the spectral transform according
		to \eqref{spec} and consider the matrix of measures $\d\varphi$ given by
        \begin{equation*}
            \d\varphi\big|_{\Sigma}:=\frac{1}{\pi}
            \begin{pmatrix}
                \rho^{\prime}&0\\
                0&\rho^{\prime}\\
            \end{pmatrix}
            \d E,\quad\quad\d\varphi\big|_{\R\setminus\Sigma}=0,
        \end{equation*}
		then we have,  for any $u\in\ell^2(\Z)$,
		\begin{equation}\label{well-defined}
			\left(1-\varepsilon_0^{\frac{\sigma^2}{10}}\right)\norm{u}^2_{\ell^2(\Z)}\leq\norm{\mathcal{S}u}_{\mathcal{L}^{2}\left(\d\varphi\right)}^2\leq\left(1+\varepsilon_0^{\frac{\sigma^2}{10}}\right)\norm{u}^2_{\ell^2(\Z)} \ ,
		\end{equation}
		and also
		\begin{equation}
			\label{infinito}
			\left| \frac{1}{\pi}\int_\Sigma \left(g_1(E){\cK}_n(E)+g_2(E){\cJ}_n(E)\right) \rho^{\prime}\d E -u_n\right|\leq
			\varepsilon_0^{\frac{\sigma^2}{10}}\norm{u}_{\ell^\infty\left(\Z\right)}.
		\end{equation}
        Furthermore, the functions ${\cK}(E)$ and ${\cJ}(E)$ have the
		following properties:
		\begin{equation}\label{K_and_J}
			{\cK}_n(E)=\sum_{n_{\Delta}=n,n\pm1}\beta_{n, n_{\Delta}}(E)\sin n_{\Delta}\rho(E) \ , \quad {\cJ}_n(E)=\sum_{n_{\Delta}=n,n\pm1}\beta_{n,n_{\Delta}}(E)\cos n_{\Delta}\rho(E) \ ,
		\end{equation}
		with $\rho$ the fibered rotation number of the cocycle $(\omega, A_0+F_0)$ and 
        \begin{equation}\label{eq:beta on SJ}
          |\beta_{n, n_{\Delta}}-\delta_{n,n_{\Delta}}|_{\Sigma_0}\leq \varepsilon_0^{\frac14} \ , \qquad |\beta_{n, n_{\Delta}}|_{\Sigma_{j+1}}\leq \varepsilon_j^\sigma, \quad j\geq 0.  
        \end{equation}
		Given any $J\in\N$, there exist $\beta^{J}_{n, n_{\Delta}}$, smooth on each connected component of $\Gamma^{(J)}_j$, satisfying
		\begin{equation}\label{esti_beta_0}
			\left|\beta^{J}_{n,n_\Delta} -\delta_{n, n_\Delta}\right|_{{\cC}^2(\Gamma^{(J)}_0)}\leq \varepsilon_{0}^{\frac14} \ ,
		\end{equation}
		and if $J\geq 1$, then
		\begin{equation}\label{esti_beta_j+10}
			|\beta_{n,n_\Delta}^{J}|_{{\cC}^1(\Gamma^{(J)}_{j+1})}\leq \varepsilon_j^{3\sigma} \ ,\quad 0\leq j\leq J-1 \ .
		\end{equation}
		Moreover,
		\begin{equation}\label{error3}
			\left|\beta_{n, n_{\Delta}}-\beta^{J}_{n, n_{\Delta}}\right|_{\Sigma_j}\leq 10\varepsilon_{J}^{\frac14} \ ,\quad 0\leq j\leq J \ .
		\end{equation}
	\end{proposition}

	\section{Dispersive estimate} \label{sec:3}
	Let 
    \begin{equation*}
        u\left(t\right)=\frac{\sin \left(t\sqrt{m^{2}+G_{\theta}}\right)}{\sqrt{m^{2}+G_{\theta}}}\,\psi+\cos \left(t\sqrt{m^{2}+G_{\theta}}\right)\,\varphi
    \end{equation*}
be the solution of the equation \eqref{eq:DKG}.	
	Set
    \begin{equation*}
        \mathcal{S}(u(t)):=\left(
        \begin{matrix}
		g_1(E,t)\\
        g_2(E,t)
	  \end{matrix}
        \right)
    \end{equation*}
    and
    \begin{equation*}
        \mathcal{S}\varphi=\left(g_{1}^{\varphi},g^{\varphi}_{2}\right)^{T},\quad\mathcal{S}\psi=\left(g_{1}^{\psi},g_{2}^{\psi}\right)^{T}
    \end{equation*}
    are the modified spectral
transforms for initial data $\varphi$ and $\psi$ respectively.
	It follows from the  functional calculus that for a.e. $E\in\Sigma$, 
	$$
	\begin{aligned}
		\begin{pmatrix} g_1(E,t) \\ g_2(E,t) \end{pmatrix}&= \cos(t\Omega(E)) \begin{pmatrix} g_1^{\varphi}(E) \\ g_2^{\varphi}(E) \end{pmatrix} + \frac{\sin(t\Omega(E))}{\Omega(E)} \begin{pmatrix} g_1^{\psi}(E) \\ g_2^{\psi}(E) \end{pmatrix}\\
		&=\frac{e^{{\rm i} \Omega(E)t}}{2} \left[ \begin{pmatrix} g_1^{\varphi} \\ g_2^{\varphi} \end{pmatrix} + \frac{1}{{\rm i}\Omega(E)} \begin{pmatrix} g_1^{\psi} \\ g_2^{\psi} \end{pmatrix} \right]+ \frac{e^{-{\rm i} \Omega(E)t}}{2} \left[ \begin{pmatrix} g_1^{\varphi} \\ g_2^{\varphi} \end{pmatrix} - \frac{1}{{\rm i}\Omega(E)} \begin{pmatrix} g_1^{\psi} \\ g_2^{\psi} \end{pmatrix} \right],
	\end{aligned}
	$$
	where $\Omega(E) := \sqrt{2+m^2+E}$.

	By \eqref{infinito}, we have
	\begin{equation}\label{eq:u(t)_int}
		|u(n,t)|\leq\frac{1}{\pi}\left|\int_\Sigma \left(g_1(E,t){\cK}_n(E)+g_2(E,t){\cJ}_n(E)\right) \rho^{\prime}\d E\right|+\varepsilon_0^{\frac{\sigma^2}{10}}\|u(\cdot,t)\|_{\ell^\infty\left(\Z\right)} \ .
	\end{equation}
	To estimate $\|u(\cdot,t)\|_{\ell^\infty\left(\Z\right)}$, it suffices to control the integral term by taking $\varepsilon_{0}$ to be small. Substituting the expression of $\mathcal{S}(u(t))$ into the integral, and recalling the definitions of $g^{\varphi}$ and $g^{\psi}$, we obtain
	\begin{equation}\label{eq:appro_u}
		\begin{aligned}
			&\int_\Sigma \left(g_1(E,t){\cK}_n(E)+g_2(E,t){\cJ}_n(E)\right)\, \rho^{\prime}\d E\\
			=&\frac{1}{2} \int_\Sigma e^{{\rm i}\Omega(E)t} \sum_{k\in\Z} \left(\varphi(k) + \frac{\psi(k)}{{\rm i}\Omega(E)}\right) \left({\cK}_k(E){\cK}_n(E)+{\cJ}_k(E){\cJ}_n(E)\right)\, \rho^{\prime}\d E \\
			&+\frac{1}{2} \int_\Sigma e^{-{\rm i}\Omega(E)t} \sum_{k\in\Z} \left(\varphi(k) - \frac{\psi(k)}{{\rm i}\Omega(E)}\right) \left({\cK}_k(E){\cK}_n(E)+{\cJ}_k(E){\cJ}_n(E)\right)\, \rho^{\prime}\d E \ .
		\end{aligned}
	\end{equation}
	Note that the term $\left({\cK}_k(E){\cK}_n(E)+{\cJ}_k(E){\cJ}_n(E)\right)$ can be expanded as
	\begin{equation}\label{eq:K,J expansions}
		\sum_{k_{\Delta}, n_{\Delta}} \beta_{k,k_{\Delta}} \beta_{n,n_{\Delta}} \cos\left((k_{\Delta}-n_{\Delta})\rho(E)\right) \ .
	\end{equation}
	Since $m^{2}>0$, the factor $\frac{1}{\Omega(E)}= \frac{1}{\sqrt{2+m^2+E}}$ is smooth and bounded on the spectrum of $G_{\theta}+m^2$. By the symmetry of the complex conjugate terms in \eqref{eq:appro_u}, the dispersive decay estimates for both integrals depend on the oscillatory behavior of the phase function. Thus, the problem reduces to establishing the estimates for the oscillatory integral associated with the kernel
	$e^{-{\rm i} t \sqrt{2+m^2+E}}$. For simplicity, we denote
	$$\mathbb{I}(S;k_{\Delta},n_{\Delta},\rho):=\int_{S} \beta_{k,k_{\Delta}}\beta_{n,n_{\Delta}}\cos
	\big((k_{\Delta}-n_{\Delta})\rho\left(E\right)\big)e^{-\mathrm{i}t\Omega\left(E\right)}\cdot\rho^{\prime}\d E, \quad \text{for}\ S\subset\R,$$ 
	and
	$$\mathbb{I}^{J}(S;k_{\Delta},n_{\Delta},\rho):=\int_{S} \beta^J_{k,k_{\Delta}} \beta^J_{n,n_{\Delta}}\cos
	\big((k_{\Delta}-n_{\Delta})\rho\left(E\right)\big)e^{-\mathrm{i}t\Omega\left(E\right)}\cdot\rho^{\prime} \d E, \quad \text{for}\ S\subset\R.$$ 
	Whenever no confusion can arise, we simply write $\mathbb{I}(E)$ and $\mathbb{I}^{J}(E)$ respectively. 
	
	To prove Theorem \ref{thm:disper}, we need the following lemma.
	\begin{lemma}\label{lem:disper-log}
		Assume $m^{2}>0$ and $|V|_r=\varepsilon_0 \leq \varepsilon_{\ast\ast}$  is sufficiently small, where $\varepsilon_{\ast\ast}$ depends on $\varepsilon_{\ast}$ in Proposition \ref{propsana} and also on $m^{2}$. Then for any $t\in\R$ and any $k,k_{\Delta},n,n_{\Delta}$,
		\begin{equation*}
			\begin{aligned}
				\left|\mathbb{I}\left(\Sigma\right)\right|\leq\frac{C\left(m\right)\left|\ln\varepsilon_{0}\right|^{82000d\left(\ln\ln\left(2+\JP{t}\right)\right)^{2}}}{\JP{t}^{\frac{1}{3}}}.
			\end{aligned}
		\end{equation*}
		Here the constant $C\left(m\right)$ depends only on $m$.
	\end{lemma}

	\begin{proof}[Proof of Theorem \ref{thm:disper} assuming Lemma \ref{lem:disper-log}]
		By the expansions \eqref{eq:K,J expansions}, we apply Lemma \ref{lem:disper-log} to equation \eqref{eq:appro_u} and derive
		\begin{equation}\label{eq:L infty estimate}
			\begin{aligned}
				&\left|\int_\Sigma \left(g_1(E,t){\cK}_n(E)+g_2(E,t){\cJ}_n(E)\right)\, \rho^{\prime}\d E\right|\\
				\leq&C\left(m\right)\cdot\frac{\left|\ln\varepsilon_{0}\right|^{82000\left(\ln\ln\left(2+\JP{t}\right)\right)^{2}d}}{\JP{t}^{\frac{1}{3}}}\left(\sum_{k\in\Z}\left|\varphi\left(k\right)\right|+\frac{1}{\inf_{E\in\Sigma}\left|\Omega\left(E\right)\right|}\sum_{k\in\Z}\left|\psi\left(k\right)\right|\right),
			\end{aligned}
		\end{equation}
		where the absolute constant $C_{1}$ relates to the estimate of $e^{{\rm i}\Omega\left(\lambda\right)t}$ and so does $C_{2}$ for the estimate related to $e^{-{\rm i}\Omega\left(\lambda\right)t}$. Since $\varepsilon_{0}<1$ is small, and by \eqref{eq:u(t)_int}, \eqref{eq:L infty estimate} we deduce
		\begin{equation*}
			\norm{u\left(t\right)}_{\ell^{\infty}\left(\Z\right)}\leq\frac{C\left(m\right)}{1-\varepsilon_0^{\frac{\sigma^2}{10}}}\frac{\left|\ln\varepsilon_{0}\right|^{82000\left(\ln\ln\left(2+\JP{t}\right)\right)^{2}d}}{\JP{t}^{\frac{1}{3}}}\left(\norm{\varphi}_{\ell^{1}\left(\Z\right)}+\frac{1}{\inf_{E\in\Sigma}\left|\Omega\left(E\right)\right|}\norm{\psi}_{\ell^{1}\left(\Z\right)}\right).
		\end{equation*}
		This completes the proof.
	\end{proof}
	
We now begin the proof of Lemma~\ref{lem:disper-log}.  
First, we state the following lemma, which asserts that $\mathbb{I}(\Sigma)$ can be approximated by $\mathbb{I}^{J}(\Sigma)$.

\begin{lemma}\label{lem:approx 1}
For any $J\geq 1$, we have
  \begin{equation*}
		\left|\mathbb{I}\left(\Sigma\right)-\mathbb{I}^{J}(\Sigma)\right|\leq32\varepsilon_{J}^{\frac{3\sigma}{4}}.
  \end{equation*}  
\end{lemma}

\begin{proof}
    
We divide the summation into two parts:
$$
\left|\mathbb{I}\left(\Sigma\right)-\mathbb{I}^{J}(\Sigma)\right|\leq\sum_{j=0}^{J}\left|\mathbb{I}\left(\Sigma_{j}\right)-\mathbb{I}^{J}\left(\Sigma_{j}\right)\right|+\sum_{j=J+1}^{\infty}\left|\mathbb{I}\left(\Sigma_{j}\right)-\mathbb{I}^{J}\left(\Sigma_{j}\right)\right|.
$$

By Proposition \ref{prop:MST}, we obtain
		\begin{align}\label{eq:es-beta 1}
			\notag   \left|\beta_{k,k_{\Delta}}\beta_{n,n_{\Delta}}-\beta^{J}_{k,k_{\Delta}}\beta^{J}_{n,n_{\Delta}}\right|_{\Sigma_{j}}\leq&\left|\beta_{k,k_{\Delta}}\left(\beta_{n,n_{\Delta}}-\beta^{J}_{n,n_{\Delta}}\right)\right|_{\Sigma_{j}}+\left|\beta^{J}_{n,n_{\Delta}}\left(\beta_{k,k_{\Delta}}-\beta^{J}_{k,k_{\Delta}}\right)\right|_{\Sigma_{j}}\\ \notag 
			\leq&10\varepsilon_{J}^{\frac{1}{4}}\left(\left(\delta_{n,n_{\Delta}}+\delta_{k,k_{\Delta}}\right)\cdot\varepsilon_{0}^{\frac{1}{4}}+\varepsilon_{0}^{\sigma}+\varepsilon_{0}^{3\sigma}\right)\\ 
			\leq&11\varepsilon^{\sigma}_{0}\varepsilon^{\frac{1}{4}}_{J}\quad{\rm for}\;0\leq j\leq J.
		\end{align}
	Here and below, we often use $C \varepsilon_j^{\lambda_1} \leq \varepsilon_j^{\lambda_2}$ for $\lambda_1>\lambda_2$.	It follows from Proposition \ref{propsana1} that  
		\begin{equation}\label{eq:rho-SigmaJ}
			\left|\rho\left(\Sigma_{j+1}\right)\right|\leq\left|\rho\left(\Gamma^{\left(J\right)}_{j+1}\right)\right|\leq3|\ln\varepsilon_j|^{2 d}\varepsilon_{j}^{\sigma},\quad{\rm for\;all}\;0\leq j\leq J-1.
		\end{equation}
		Combining \eqref{eq:es-beta 1} and \eqref{eq:rho-SigmaJ}, we derive 
		\begin{equation*}
			\begin{aligned}
				&\sum_{j=0}^{J}\left|\mathbb{I}\left(\Sigma_{j}\right)-\mathbb{I}^{J}\left(\Sigma_{j}\right)\right|\\
				=&\sum_{j=0}^{J}\left|\int_{\Sigma_{j}}\left(\beta_{k,k_{\Delta}}\beta_{n,n_{\Delta}}-\beta^{J}_{k,k_{\Delta}}\beta^{J}_{n,n_{\Delta}}\right)\cos\big((k_{\Delta}-n_{\Delta})\rho\left(E\right)\big)e^{-{\rm i}t\Omega\left(E\right)}\cdot\rho^{\prime}\d E\right|\\
				\leq&11\varepsilon^{\sigma}_{0}\varepsilon^{\frac{1}{4}}_{J}\sum_{j=0}^{J}\int_{\Sigma_{j}}\rho^{\prime}\d E\quad\left({\rm noting\;that\;}\left|\cos\big((k_{\Delta}-n_{\Delta})\rho\left(E\right)\big)e^{-{\rm i}t\Omega\left(E\right)}\right|\leq 1\right)\\
				\leq&11\varepsilon^{\sigma}_{0}\varepsilon^{\frac{1}{4}}_{J}\left(\pi+\sum_{j=0}^{J-1}3|\ln\varepsilon_j|^{2 d}\varepsilon_{j}^{\sigma}\right)\\
				\leq&11\pi\varepsilon^{\sigma}_{0}\varepsilon^{\frac{1}{4}}_{J}\left(1+\varepsilon_{0}^{\frac{3\sigma}{4}}\right).
			\end{aligned}
		\end{equation*}
The last inequality follows from the geometric-series bound
\[
\sum_{j=0}^{\infty}3|\ln\varepsilon_j|^{2 d}\varepsilon_{j}^{\sigma} 
\le \frac{3 |\ln \varepsilon_0|^{2d} \varepsilon_0^\sigma}{1 - (1+\sigma)^{2d} \varepsilon_0^{\sigma^2}},
\]
provided $\varepsilon_0$ is chosen sufficiently small so that
\[
\sum_{j=0}^{\infty}3|\ln\varepsilon_j|^{2 d}\varepsilon_{j}^{\sigma}
\le\pi\varepsilon_{0}^{\frac{3\sigma}{4}}.
\]

		For any $j\geq J+1$, we apply the $L^{\infty}$-estimate of $\beta^{J}_{k,k_{\Delta}}\beta^{J}_{n,n_{\Delta}}$ on $\Sigma_{j}$ and Proposition \ref{prop:MST}, 
		\begin{equation}\label{eq:es-beta 2}
			\begin{aligned}
				\left|\beta_{k,k_{\Delta}}\beta_{n,n_{\Delta}}-\beta^{J}_{k,k_{\Delta}}\beta^{J}_{n,n_{\Delta}}\right|_{\Sigma_{j}}\leq\varepsilon_{J}^{2\sigma}+\left(\delta_{k,k_{\Delta}}+\varepsilon_{0}^{\frac{1}{4}}+\varepsilon_{0}^{3\sigma}\right)\left(\delta_{n,n_{\Delta}}+\varepsilon_{0}^{\frac{1}{4}}+\varepsilon_{0}^{3\sigma}\right)\leq 10.
			\end{aligned}
		\end{equation}
		
		We calculate that
		\begin{equation*}
			\begin{aligned}
				&\sum_{j=J+1}^{\infty}\left|\mathbb{I}\left(\Sigma_{j}\right)-\mathbb{I}^{J}\left(\Sigma_{j}\right)\right|\\
				=&\sum_{j=J+1}^{\infty}\left|\int_{\Sigma_{j}}\left(\beta_{k,k_{\Delta}}\beta_{n,n_{\Delta}}-\beta^{J}_{k,k_{\Delta}}\beta^{J}_{n,n_{\Delta}}\right)\cos\big((k_{\Delta}-n_{\Delta})\rho\left(E\right)\big)e^{-{\rm i}t\Omega\left(E\right)}\cdot\rho^{\prime}\d E\right|\\
				\leq&10\sum_{j=J+1}^{\infty}\int_{\Sigma_{j}}\rho^{\prime}\d E\\
                \leq&10\sum_{j=J+1}^{\infty}\left|\rho\left(\Sigma_{j}\right)\right|\qquad\left({\rm by\;inequality\; \eqref{eq:rho-SigmaJ}}\right)\\
				\leq&10\pi\varepsilon_{J}^{\frac{3\sigma}{4}}.\\
			\end{aligned}
		\end{equation*}
        Combining all error terms, 
        \begin{equation*}
            \left|\mathbb{I}\left(\Sigma\right)-\mathbb{I}^{J}(\Sigma)\right|\leq11\pi\varepsilon^{\sigma}_{0}\varepsilon^{\frac{1}{4}}_{J}\left(1+\varepsilon_{0}^{\frac{3\sigma}{4}}\right)+10\pi\varepsilon_{J}^{\frac{3\sigma}{4}}\leq32\varepsilon_{J}^{\frac{3\sigma}{4}}.
        \end{equation*}
        This completes the proof.	
\end{proof}

For simplicity, we denote $\Theta_1:=\sup\Sigma-\inf\Sigma$ and $\Theta_2:=\inf_{E \in \Sigma}\Omega\left(E\right)$.

In the case of $|k_{\Delta}-n_{\Delta}|$ is large, we derive the following bound.      
\begin{lemma}\label{lem:large}
For any $t \in \R$, $k_{\Delta}-n_{\Delta}\neq 0$, the following estimate holds, 
\begin{equation*}
    \left|\mathbb{I}^{J}(\Sigma)\right|\leq\frac{18}{\left|k_{\Delta}-n_{\Delta}\right|}\left|\ln \varepsilon_{0}\right|^{2J^{2}d}+\frac{3\Theta_{1}\JP{t}}{\left|k_{\Delta}-n_{\Delta}\right|}\left(\left(1+\varepsilon_{0}^{\frac{1}{4}}\right)^{2}+\varepsilon_{J}^{6\sigma}\right)\left(1+\frac{1}{\Theta_{2}}\right).\\
\end{equation*}		
\end{lemma}
\begin{proof}
By Proposition~\ref{propsana1}, for each $0\leq j\leq J$ we can write
\begin{equation*}
  \Gamma_{j}^{(J)}
  =\mathcal{R}_{j}\cup\Bigl(\,\bigsqcup_{s^{J}_{j}}I_{s^{J}_{j}}\Bigr),
\end{equation*}
where $\mathcal{R}_{j}$ is a finite set of points (possibly empty) and each $I_{s^{J}_{j}}$ is an open interval in $\Gamma_{j}^{(J)}$. Recall that $\rho$ is absolutely continuous. Then
		\begin{equation}\label{eq:Ij-M large}
			\begin{aligned}
				\mathbb{I}^{J}(\Sigma)=&\int_{\inf\Sigma}^{\sup\Sigma}\beta^J_{k,k_{\Delta}} \beta^J_{n,n_{\Delta}}\cos\big((k_{\Delta}-n_{\Delta})\rho\left(E\right)\big)e^{-{\rm i}t\Omega\left(E\right)}\cdot\rho^{\prime} \d E\\
				=&\frac{1}{k_{\Delta}-n_{\Delta}}\sum_{j=0}^{J}\sum_{s_{j}^{J}}\big(\beta^J_{k,k_{\Delta}} \beta^J_{n,n_{\Delta}}\sin\big((k_{\Delta}-n_{\Delta})\rho\left(E\right)\big)e^{-{\rm i}t\Omega\left(E\right)}\big)\Big|_{\inf I_{s_{j}^{J}}}^{\sup I_{s_{j}^{J}}}\\
				&\quad-\frac{1}{k_{\Delta}-n_{\Delta}}\sum_{j=0}^{J}\sum_{s_{j}^{J}}\int_{I_{s_{j}^{J}}}\big(\beta^J_{k,k_{\Delta}} \beta^J_{n,n_{\Delta}}e^{-{\rm i}t\Omega\left(E\right)}\big)^{\prime}\sin\big((k_{\Delta}-n_{\Delta})\rho\left(E\right)\big)\d E.\\
			\end{aligned}
		\end{equation}
		Using Proposition \ref{propsana1} again, we derive the upper bound for connected components of $\bigcup \Gamma_{j}^{(J)}$:
		\begin{equation*}
			\sum_{j=0}^{J}\#\left(s_{j}^{J}\right)\leq \left|\ln \varepsilon_{0}\right|^{2J^{2}d}.
		\end{equation*}
		For the first term  in equation \eqref{eq:Ij-M large}, by the $L^{\infty}$-estimate and inequality \eqref{esti_beta_0} that
		\begin{equation}\label{eq:first term}
			\begin{aligned}
				&\left|\frac{1}{k_{\Delta}-n_{\Delta}}\sum_{j=0}^{J}\sum_{s_{j}^{J}}\big(\beta^J_{k,k_{\Delta}} \beta^J_{n,n_{\Delta}}\sin\big((k_{\Delta}-n_{\Delta})\rho\left(E\right)\big)e^{-{\rm i}t\Omega\left(E\right)}\big)\Big|_{\inf I_{s_{j}^{J}}}^{\sup I_{s_{j}^{J}}}\right|\\
				\leq&\frac{2}{\left|k_{\Delta}-n_{\Delta}\right|}\left|\ln \varepsilon_{0}\right|^{2J^{2}d}\cdot\left(1+\varepsilon_{0}^{\frac{1}{4}}+\varepsilon_{0}^{3\sigma}\right)^{2}\\
                \leq&\frac{18}{\left|k_{\Delta}-n_{\Delta}\right|}\left|\ln \varepsilon_{0}\right|^{2J^{2}d}.
			\end{aligned}
		\end{equation}
For the second term in \eqref{eq:Ij-M large}, we need the following direct estimate
\begin{equation}\label{eq:betac1}
	\left|\beta^J_{k,k_{\Delta}} \beta^J_{n,n_{\Delta}}\right|_{\mathcal{C}^{1}\left(\Gamma^{\left(J\right)}_{j}\right)}\leq
	\begin{cases}
			3\left(1+\varepsilon_{0}^{\frac{1}{4}}\right)^{2}\;{\rm when}\;j=0;\\
				3\varepsilon_{j}^{6\sigma}\qquad{\rm when}\;1\leq j\leq J.\\
		\end{cases}
\end{equation}
        Using inequality
        \eqref{eq:betac1}, we deduce

		\begin{equation*}\label{eq:es-M small-second term}
		\begin{aligned}
            &\left|\big(\beta^J_{k,k_{\Delta}} \beta^J_{n,n_{\Delta}}e^{-{\rm i}t\Omega\left(E\right)}\big)^{\prime}\right|_{\left[\inf\Sigma,\sup\Sigma\right]}\\
            =&\left|\left(\beta^J_{k,k_{\Delta}} \beta^J_{n,n_{\Delta}}\right)^{\prime}e^{-{\rm i}t\Omega\left(E\right)}-{\rm i}t\Omega^{\prime}\left(E\right)\beta^J_{k,k_{\Delta}} \beta^J_{n,n_{\Delta}}e^{-{\rm i}t\Omega\left(E\right)}\right|_{\left[\inf\Sigma,\sup\Sigma\right]}\\
            \leq&3\left(\left(1+\varepsilon_{0}^{\frac{1}{4}}\right)^{2}+\varepsilon_{j}^{6\sigma}\right)\left(1+\frac{\JP{t}}{\Theta_{2}}\right).
		\end{aligned}
		\end{equation*}
		 Therefore,
		\begin{equation}\label{eq:second term}
			\begin{aligned}
				&\left|\frac{1}{k_{\Delta}-n_{\Delta}}\sum_{j=0}^{J}\sum_{s_{j}^{J}}\int_{I_{s_{j}^{J}}}\big(\beta^J_{k,k_{\Delta}} \beta^J_{n,n_{\Delta}}e^{-{\rm i}t\Omega\left(E\right)}\big)^{\prime}\sin\big((k_{\Delta}-n_{\Delta})\rho\left(E\right)\big)\d E\right|\\
                \leq&\frac{3\Theta_{1}\JP{t}}{\left|k_{\Delta}-n_{\Delta}\right|}\left(\left(1+\varepsilon_{0}^{\frac{1}{4}}\right)^{2}+\varepsilon_{j}^{6\sigma}\right)\left(1+\frac{1}{\Theta_{2}}\right).\\
			\end{aligned}
		\end{equation}
		Combining \eqref{eq:first term} and \eqref{eq:second term} we complete the proof.
\end{proof}

Unlike the case where $|k_{\Delta}-n_{\Delta}|$ is large, the estimate becomes  more complex when $|k_{\Delta}-n_{\Delta}|$ is small. In the following lemma, we reduce the estimation of $\mathbb{I}^{J}\left(\Sigma\right)$ to the restricted sum of integrals over the set where the phase derivatives are well-behaved.

\begin{lemma}\label{lem:small+error}
Set 
\begin{equation*}
    \Lambda_{j}^{\left(J\right)}:=\left\{E\in\Gamma_{j}^{\left(J\right)}:\left|\sin \xi(E)\right|\geq \varepsilon_{J}^{\frac{1}{20}}\right\}\quad{\rm for}\;0\leq j\leq J.
\end{equation*} 
Then
\begin{equation*}
    \left|\mathbb{I}^{J}\left(\Sigma\right)-\sum_{j=0}^{J}\mathbb{I}^{J}(\Lambda^{\left(J\right)}_{j};k_{\Delta},n_{\Delta},\rho_{J})\right|\leq61\varepsilon_{J}^{\frac{3\sigma}{4}}+100\left|k_{\Delta}-n_{\Delta}\right|\varepsilon^{\frac{1}{10}}_{J}.
\end{equation*}     
\end{lemma}

\begin{proof}
 	 By the definitions of $\mathbb{I}^{J}\left(\Sigma\right)$ and $\sum_{j=0}^{J}\mathbb{I}^{J}(\Lambda^{\left(J\right)}_{j};k_{\Delta},n_{\Delta},\rho_{J})$, we have
		\begin{equation*}
			\begin{aligned}
				&\left|\mathbb{I}^{J}\left(\Sigma\right)-\sum_{j=0}^{J}\mathbb{I}^{J}(\Lambda^{\left(J\right)}_{j};k_{\Delta},n_{\Delta},\rho_{J})\right|\\
				&\leq\left|\sum_{j=J+1}^{\infty}\mathbb{I}^{J}\left(\Sigma_{j}\right)\right|\\
				&\quad+\left|\sum_{j=0}^{J}\int_{\Sigma_{j}\cap\Lambda_{j}^{\left(J\right)}}\beta^J_{k,k_{\Delta}} \beta^J_{n,n_{\Delta}}e^{-{\rm i}t\Omega}\Big(\rho^{\prime}\cos\big((k_{\Delta}-n_{\Delta})\rho\big)-\rho_{J}^{\prime}\cos\big((k_{\Delta}-n_{\Delta})\rho_{J}\Big)\d E \right|\\
				&\quad+\left|\sum_{j=0}^{J}\int_{\Sigma_{j}\setminus\Lambda_{j}^{\left(J\right)}}\beta^J_{k,k_{\Delta}} \beta^J_{n,n_{\Delta}}e^{-{\rm i}t\Omega}\rho^{\prime}\cos\big((k_{\Delta}-n_{\Delta})\rho\Big)\d E \right|\\
				&\quad+\left|\sum_{j=0}^{J}\int_{\Lambda_{j}^{\left(J\right)}\setminus\Sigma_{j}}\beta^J_{k,k_{\Delta}} \beta^J_{n,n_{\Delta}}e^{-{\rm i}t\Omega}\rho_{J}^{\prime}\cos\big((k_{\Delta}-n_{\Delta})\rho_{J}\Big)\d E \right|\\
				&:=\mathbf{I}_{1}+\mathbf{I}_{2}+\mathbf{I}_{3}+\mathbf{I}_{4}.
			\end{aligned}
		\end{equation*}
		Below we give the estimates for $\{\mathbf{I}_{l}\}_{1\leq l\leq 4}$.
		
		{\bf a. Estimate  $\mathbf{I}_{1}$}
		
		Combining \eqref{eq:beta on SJ} and \eqref{eq:es-beta 2}, we have 
        \begin{equation*}
            \left|\beta^{J}_{k,k_{\Delta}}\beta^{J}_{n,n_{\Delta}}\right|_{\Sigma_{j}}\leq\left(1+\varepsilon_{0}^{\frac{1}{4}}+\varepsilon_{0}^{\sigma}\right)^{2}+10\leq 19\quad{\rm for}\;j\geq J+1.
        \end{equation*}
        By \eqref{eq:rho-SigmaJ}, we deduce
		\begin{equation*}
			\mathbf{I}_{1}\leq\sum_{j=J}^{\infty}19\cdot3\left|\ln \varepsilon_{j}\right|^{2d}\varepsilon_{j}^{\sigma}\leq60\varepsilon_{J}^{\frac{3\sigma}{4}}.
		\end{equation*}
		
		{\bf b. Estimate  $\mathbf{I}_{2}$}
		
        We need to estimate
		\begin{equation}
			\begin{aligned}
				&\left|\rho^{\prime}\cos\big((k_{\Delta}-n_{\Delta})\rho\big)-\rho_{J}^{\prime}\cos\big((k_{\Delta}-n_{\Delta})\rho_{J}\big)\right|\\
				\leq&\left|\rho^{\prime}-\rho^{\prime}_{J}\right|+\left|\rho^{\prime}_{J}\big(\cos\left(\left(k_{\Delta}-n_{\Delta}\right)\rho_{J}\right)-\cos\left(\left(k_{\Delta}-n_{\Delta}\right)\rho\right)\big)\right|\\
				:=&\mathbf{M}_{1}+\mathbf{M}_{2}.\\
			\end{aligned}
		\end{equation}
		For any $0\leq j\leq J$, on $\Sigma_{j}\cap\Lambda_{j}^{\left(J\right)}$, we derive
		\begin{equation*}
			\left|\sin\xi_{J}\right|\geq\left|\sin \xi\right|-\left|\sin \xi-\sin\xi_{J}\right|\geq\frac{1}{2}\varepsilon_{J}^{\frac{1}{20}},
		\end{equation*}
        by Proposition \ref{propsana1} ({\bf S2}). On $\Sigma_{j}\cap\Lambda_{j}^{\left(J\right)}$, using Proposition \ref{propsana1} ({\bf S2}) ({\bf S3}), 
        
        we have
		\begin{equation}\label{eq:M1}
			\begin{aligned}
				\mathbf{M}_{1}=&\left|\frac{\left({\rm tr}A_{J}\right)^{\prime}}{2\sin\xi_{J}}-\frac{\left({\rm tr}B\right)^{\prime}}{2\sin\xi}\right|\\
				\leq&\frac{1}{2\left|\sin\xi_{J}\sin\xi\right|}\left(\left|\left({\rm tr}A_{J}\right)^{\prime}\left(\sin\xi-\sin\xi_{J}\right)\right|+\left|\left(\left({\rm tr}A_{J}\right)^{\prime}-\left({\rm tr}B\right)^{\prime}\right)\sin\xi_{J}\right|\right)\\
				\leq&\varepsilon_{J}^{-\frac{1}{10}}\left(1\cdot\varepsilon_{J}^{\frac{1}{4}}+2N_{j}^{10\eta}\cdot\varepsilon_{J}^{\frac{1}{4}}\right)\\
				\leq&\varepsilon^{\frac{1}{10}}_{J},
			\end{aligned}
		\end{equation}
		and 
		\begin{equation}\label{eq:M2}
			\begin{aligned}
				\mathbf{M}_{2}
=&
\left|
\rho^{\prime}_{J}
\bigl(
\cos((k_{\Delta}-n_{\Delta})\rho_J)
-
\cos((k_{\Delta}-n_{\Delta})\rho)
\bigr)
\right|.\\
				\leq&\left|\frac{\left({\rm tr}A_{J}\right)^{\prime}}{2\sin\xi_{J}}\right|\cdot2\left|\sin\Big(\frac{1}{2}\left(k_{\Delta}-n_{\Delta}\right)\left(\rho_{J}-\rho\right)\Big)\right|\\
				\leq&4N_{j}^{10\eta}\varepsilon_{J}^{-\frac{1}{20}}\left|k_{\Delta}-n_{\Delta}\right|\varepsilon^{\frac{1}{4}}_{J}.
			\end{aligned}
		\end{equation}

		By \eqref{esti_beta_0} and \eqref{esti_beta_j+10}, we have
		\begin{equation}\label{eq:betaJj}
			\left|\beta^J_{k,k_{\Delta}} \beta^J_{n,n_{\Delta}}\right|_{\Gamma^{\left(J\right)}_{j}}\leq \left(1+\varepsilon_{0}^{\frac{1}{4}}+\varepsilon_{0}^{3\sigma}\right)^{2}\leq 9\quad{\rm for}\;0\leq j\leq J.
		\end{equation}
Combining \eqref{eq:betaJj} with \eqref{eq:M1} and \eqref{eq:M2}, we derive
		\begin{equation*}
			\begin{aligned}
				\mathbf{I}_{2}\leq&\sum_{j=0}^{J}\int_{\Sigma_{j}\cap\Lambda^{\left(J\right)}_{j}}9\left(\varepsilon^{\frac{1}{10}}_{J}+4N_{j}^{10\eta}\varepsilon_{J}^{-\frac{1}{20}}\left|k_{\Delta}-n_{\Delta}\right|\varepsilon^{\frac{1}{4}}_{J}\right)\\
				\leq&9\Theta_{1}\varepsilon_{J}^{\frac{1}{10}}\left(1+\left|k_{\Delta}-n_{\Delta}\right|\right).
			\end{aligned}
		\end{equation*}
		
		{\bf c. Estimate  $\mathbf{I}_{3}$}
		
		For any $0\leq j\leq J$, on $\Sigma_{j}$, we have $\left|\xi-\xi_{J}\right|\leq\varepsilon^{\frac{1}{4}}_{J}$ via the fact in Proposition \ref{propsana1} ({\bf S2}). Moreover, on $\Sigma_{j}\setminus\Lambda_{j}^{\left(J\right)}$, we have $\left|\sin\xi\right|\leq\varepsilon^{\frac{1}{20}}_{J}$, which implies that
		\begin{equation*}
			\left|\sin\xi_{J}\right|\leq\left|\sin\xi_{j}-\sin\xi\right|+\left|\sin\xi\right|\leq\varepsilon^{\frac{1}{4}}_{J}+\varepsilon^{\frac{1}{20}}_{J}\leq\frac{3}{2}\varepsilon^{\frac{1}{20}}_{J}.
		\end{equation*}
		Thus combining \eqref{eq:betaJj} and Proposition \ref{propsana1} (S4), we have
		\begin{equation*}
			\mathbf{I}_{3}\leq 9\sum_{j=0}^{J}\int_{\Sigma_{j}\setminus\Lambda_{j}^{\left(J\right)}}\rho^{\prime}\d E\leq 9\rho(\{(\inf\Sigma,\sup\Sigma): |\sin\xi_{J}|\leq\frac32 \varepsilon_J^{\frac{1}{20}} \})\leq 9\varepsilon_{J}^{\frac{1}{24}}.
		\end{equation*}
		
		{\bf d. Estimate $\mathbf{I}_{4}$}
		
		By Proposition \ref{propsana1} (S4), for any $0\leq j\leq J$, we have $\rho\left(\Gamma^{\left(J\right)}_{j}\setminus\Sigma_{j}\right)\leq\varepsilon^{\frac{7\sigma}{8}}_{J}$. Applying \eqref{eq:betaJj}, we get
		\begin{equation*}
			\mathbf{I}_{4}\leq9\sum_{j=0}^{J}\rho_{J}\left(\Gamma^{\left(J\right)}_{j}\setminus\Sigma_{j}\right)\leq9\left(J+1\right)\varepsilon^{\frac{7\sigma}{8}}_{J}.
		\end{equation*}
		
		{\bf e. Back to estimate  $\left|\mathbb{I}^{J}\left(\Sigma\right)-\sum_{j=0}^{J}\mathbb{I}^{J}(\Lambda^{\left(J\right)}_{j};k_{\Delta},n_{\Delta},\rho_{J})\right|$}
		
		By the upper bound for $\{\mathbf{I}_{l}\}_{1\leq l\leq 4}$, we obtain
		\begin{equation*}
			\begin{aligned}
				&\left|\mathbb{I}^{J}\left(\Sigma\right)-\sum_{j=0}^{J}\mathbb{I}^{J}(\Lambda^{\left(J\right)}_{j};k_{\Delta},n_{\Delta},\rho_{J})\right|\\
				\leq&60\varepsilon_{J}^{\frac{3\sigma}{4}}+9\Theta_{1}\varepsilon_{J}^{\frac{1}{10}}\left(1+\left|k_{\Delta}-n_{\Delta}\right|\right)+9\varepsilon_{J}^{\frac{1}{24}}+9\left(J+1\right)\varepsilon^{\frac{7\sigma}{8}}_{J}\\
				\leq&61\varepsilon_{J}^{\frac{3\sigma}{4}}+100\left|k_{\Delta}-n_{\Delta}\right|\varepsilon^{\frac{1}{10}}_{J}\quad\left({\rm by}\;J\varepsilon_{J}^{\frac{7}{8}\sigma}\leq\varepsilon_{J}^{\frac{13}{16}\sigma}\right).
			\end{aligned}
		\end{equation*} 
        This completes the proof.
\end{proof}

We set 
\begin{equation*}
    \widetilde{\Lambda}^{(J)}_{j}:=\left\{E\in\Gamma^{\left(J\right)}_{j}:\left|\sin\xi_{J}\right|>\frac{3}{2}\varepsilon^{\frac{1}{20}}_{J}\right\}\quad{\rm for}\;0\leq j\leq J.
\end{equation*}
and denote 
\begin{equation*}
    S_{1}:=\bigcup_{j=0}^{J}\Lambda^{\left(J\right)}_{j}\setminus \bigcup_{j=0}^{J}\widetilde{\Lambda}^{\left(J\right)}_{j}\quad,\quad S_{2}:=\bigcup_{j=0}^{J}\Lambda^{\left(J\right)}_{j}\cap \bigcup_{j=0}^{J}\widetilde{\Lambda}^{\left(J\right)}_{j}\;.
\end{equation*}
For the corresponding integral on $S_1$, we have the following lemma.

\begin{lemma}\label{lem:S1}
For any $t\in\mathbb R$,
\begin{equation*}
\left|\mathbb{I}^{J}(S_{1};k_{\Delta},n_{\Delta},\rho_{J})\right|
\le 54|\ln\varepsilon_{0}|^{2J^{2}d}\,\varepsilon_{J}^{\frac{1}{20}}.
\end{equation*}
\end{lemma}

\begin{proof}
By definition, we have
\[
S_1 
\subset \Bigl\{E\in \bigcup_{j=0}^J\Gamma_j^{(J)}:\ 
|\sin\xi_J(E)|\le \frac{3}{2}\varepsilon_J^{\frac{1}{20}}\Bigr\}.
\]
Hence, by \eqref{eq:betaJj},
\begin{align*}
\left|\mathbb{I}^{J}(S_{1};k_{\Delta},n_{\Delta},\rho_{J})\right|
&=\left|\int_{S_1}\beta^J_{k,k_{\Delta}}(E)\beta^J_{n,n_{\Delta}}(E)\,
\cos\bigl((k_{\Delta}-n_{\Delta})\rho_{J}(E)\bigr)\,e^{-{\rm i}t\Omega(E)}\,
\rho_J'(E)\,\d E\right|\\
&\le9\int_{S_1}\rho_J'(E)\,\d E.
\end{align*}
Now we decompose $\bigcup_{j=0}^J\Gamma_j^{(J)}$ into its (open) connected components:
\[
\bigcup_{j=0}^J\Gamma_j^{(J)}=\mathcal R \sqcup \Bigl(\bigsqcup_{\nu}\mathcal L_\nu\Bigr),
\]
where $\mathcal R$ is a finite set and each $\mathcal L_\nu$ is an open interval. On each $\mathcal L_\nu$, the function $\rho_J$ is smooth and strictly monotone, and moreover
$\rho_J=\xi_J+\sum_{l=0}^{J-1}\langle k_l\rangle$ with the sum constant on $\mathcal L_\nu$.
Hence on $\mathcal L_\nu$ the condition $|\sin\xi_J(E)|\le \frac{3}{2}\varepsilon_J^{\frac{1}{20}}$
is equivalent to
\[
|\sin(\rho_J(E)-c_\nu)|\le \frac{3}{2}\varepsilon_J^{\frac{1}{20}}
\]
for some constant $c_\nu$.
Perform the change of variables $y=\rho_J(E)$ on $\mathcal L_\nu$, then
\[
\int_{S_1\cap \mathcal L_\nu}\rho_J'(E)\,\d E
=|\rho_J(S_1\cap \mathcal L_\nu)|
\le \left|\left\{y\in \rho_J(\mathcal L_\nu):\ |\sin(y-c_\nu)|\le \frac{3}{2}\varepsilon_J^{\frac{1}{20}}\right\}\right|.
\]
For any interval $I\subset\mathbb R$ and any $\delta\in(0,1)$, one has the elementary estimate
\[
\left|\{y\in I:\ |\sin y|\le \delta\}\right|\le 4\delta,
\]
since on each period $[0,\pi]$ the set $\{|\sin y|\le\delta\}$ consists of two subintervals
near $0$ and $\pi$ of total length $\le 4\delta$.
Applying this with $\delta=\frac{3}{2}\varepsilon_J^{\frac{1}{20}}$ gives
\[
\int_{S_1\cap \mathcal L_\nu}\rho_J'(E)\,\d E \le 6\,\varepsilon_J^{\frac{1}{20}}.
\]
Summing over all components,
\[
\int_{S_1}\rho_J'(E)\,\d E
\le \sum_{\nu}\int_{S_1\cap \mathcal L_\nu}\rho_J'(E)\,\d E
\le 6|\ln\varepsilon_0|^{2J^2 d}\,\varepsilon_J^{\frac{1}{20}}.
\]
Thus,
\[
\left|\mathbb{I}^{J}(S_{1};k_{\Delta},n_{\Delta},\rho_{J})\right|
\le 54\,|\ln\varepsilon_0|^{2J^2 d}\,\varepsilon_J^{\frac{1}{20}}.
\]
This completes the proof.
\end{proof}

Note that there are at most $2$ intervals contained in each connected component of $\bigcup_{j=0}^{J}\Lambda^{\left(J\right)}_{j}$. For the corresponding integral on $S_2$, we obtain estimate as follows.

\begin{lemma}\label{lem:S2}
    For any $J\geq 1$, there exists a constant $C(m)$ such that
    \begin{equation*}
        \left|\mathbb{I}^{J}(S_{2};k_{\Delta},n_{\Delta},\rho_{J})\right|\leq C\left(m\right)\left|\ln\varepsilon_{0}\right|^{2J^{2}d}\JP{t}^{-\frac{1}{3}}.
    \end{equation*}
\end{lemma}
\begin{proof}
    We change variables first. By Proposition \ref{propsana1} ({\bf S3}), the set
		\begin{equation*}
			\bigcup_{j=0}^{J}\left\{E\in \Gamma_{j}^{\left(J\right)}:\left|\sin\xi_{J}\right|\geq \frac{3}{2}\varepsilon^{\frac{1}{20}}_{J}\right\}
		\end{equation*}
		consists of at most $2\left|\ln\varepsilon_{0}\right|^{2J^{2}d}$ connected components. Here and below, we set
		\begin{equation*}
			\bigcup_{j=0}^{J}\left\{E\in \Gamma_{j}^{\left(J\right)}:\left|\sin\xi_{J}\right|\geq \frac{3}{2}\varepsilon^{\frac{1}{20}}_{J}\right\}=\mathcal{T}\sqcup\left(\bigsqcup_{\kappa}\mathcal{L}_{\kappa}\right),
		\end{equation*}
		where $\mathcal{T}$ is a finite subset (perhaps empty) and $\mathcal{L}_{\kappa}$ are open intervals. By definition, we have
		\begin{equation*}
			\sum_{j=0}^{J}\mathbb{I}^{J}(\Lambda^{\left(J\right)}_{j};k_{\Delta},n_{\Delta},\rho_{J})=\sum_{\kappa}\mathbb{I}^{J}(\mathcal{L}_{\kappa};k_{\Delta},n_{\Delta},\rho_{J}),
		\end{equation*}
		and for any fixed $\kappa$, 
		\begin{equation*}
			\mathbb{I}^{J}(\mathcal{L}_{\kappa};k_{\Delta},n_{\Delta},\rho_{J})=\int_{\mathcal{L}_{\kappa}}\beta^J_{k,k_{\Delta}} \beta^J_{n,n_{\Delta}}\cos
			\big((k_{\Delta}-n_{\Delta})\rho_{J}\left(E\right)\big)e^{-{\rm i}t\Omega\left(E\right)}\cdot\rho_{J}^{\prime} \d E.
		\end{equation*}
		We denote the inverse of $\rho_{J}=\rho_{J}\left(E\right)$ by
		\begin{equation*}
		  E=E\left(\rho_{J}\right)\quad{\rm on\;interval}\;\left(\rho_{J}\left(\inf\mathcal{L}_{\kappa}\right),\rho_{J}\left(\sup\mathcal{L}_{\kappa}\right)\right),
		\end{equation*}
since in each $\mathcal{L}_{\kappa}$, $\rho_{J}$ is strictly increasing. Thus, we derive
		\begin{equation}\label{eq:OI}
			\begin{aligned}
				&\mathbb{I}^{J}(\mathcal{L}_{\kappa}\cap S_{2};k_{\Delta},n_{\Delta},\rho_{J})\\
				=&\int_{\rho_{J}\left(\mathcal{L}_{\kappa}\cap S_{2}\right)}\beta^J_{k,k_{\Delta}}\left(E\left(\rho_{J}\right)\right)\beta^J_{n,n_{\Delta}}\left(E\left(\rho_{J}\right)\right)\cos
				\big((k_{\Delta}-n_{\Delta})\rho_{J}\big)e^{-{\rm i}t\Omega\left(E\left(\rho_{J}\right)\right)}\d\rho_{J}\\
				=&\frac{1}{2}\int_{\rho_{J}\left(\mathcal{L}_{\kappa}\cap S_{2}\right)}\beta^J_{k,k_{\Delta}}\left(E\left(\rho_{J}\right)\right)\beta^J_{n,n_{\Delta}}\left(E\left(\rho_{J}\right)\right)e^{-{\rm i}t\Omega\left(E\left(\rho_{J}\right)\right)}\left(e^{{\rm i}(k_{\Delta}-n_{\Delta})\rho_{J}}+e^{-{\rm i}(k_{\Delta}-n_{\Delta})\rho_{J}}\right)\d\rho_{J}.
			\end{aligned}
		\end{equation}
        
		Now we deal with the oscillatory integrals. We briefly analyze the amplitude function. Using \eqref{esti_beta_0} and \eqref{esti_beta_j+10}, we derive 
		\begin{equation*}
			\left|\beta^{J}_{k,k_{\Delta}}\right|_{\mathcal{C}^{1}\left(\Gamma_{0}^{(J)}\right)}\leq 2\quad{\rm and}\quad\left|\beta^{J}_{k,k_{\Delta}}\right|_{\mathcal{C}^{1}\left(\Gamma_{j+1}^{(J)}\right)}\leq\varepsilon_{j}^{3\sigma}\quad{\rm for}\;0\leq j\leq J-1.
		\end{equation*}
		and by the fact that $\left|\frac{\d \rho_{J}}{\d E}\right|>\frac{1}{3}$, we have $\left|\frac{\d E}{\d\rho_{J}}\right|<3$. Hence the chain law implies that
		\begin{equation*}
			\left|\beta^{J}_{k,k_{\Delta}}\left(E\left(\cdot\right)\right)\right|_{\mathcal{C}^{1}\left(\rho_{J}\left(\Gamma_{0}^{(J)}\right)\right)}\leq 6\quad{\rm and}\quad\left|\beta^{J}_{k,k_{\Delta}}\left(E\left(\cdot\right)\right)\right|_{\mathcal{C}^{1}\left(\rho_{J}\left(\Gamma_{j+1}^{(J)}\right)\right)}\leq3\varepsilon_{j}^{3\sigma}\quad{\rm for}\;0\leq j\leq J-1.
		\end{equation*}
		For any $\kappa$, we always have
		\begin{equation*}
			\left|\beta^J_{k,k_{\Delta}}\left(E\left(\cdot\right)\right)\beta^J_{n,n_{\Delta}}\left(E\left(\cdot\right)\right)\right|_{\mathcal{C}^{1}\left(\rho_{J}\left(\mathcal{L}_{\kappa}\right)\right)}\leq 108.
		\end{equation*}

		Recall \eqref{eq:OI}; without loss of generality, we consider only the following phase function.
		\begin{equation*}
			\mathcal{H}_{t}\left(\rho_{J}\right):=t\Omega\left(E\left(\rho_{J}\right)\right)+(k_{\Delta}-n_{\Delta})\rho_{J}\quad{\rm on}\;\rho_{J}\left(\mathcal{L}_{\kappa}\right)
		\end{equation*}
		for some fixed $\kappa$. A direct calculation shows that
		\begin{equation}\label{eq:Ht's derivatives}
			\begin{cases}
				\frac{\d}{\d\rho_{J}}\mathcal{H}_{t}\left(\rho_{J}\right)=\frac{t}{2\Omega(E\left(\rho_{J}\right))}\frac{\d E}{\d\rho_{J}}+k_{\Delta}-n_{\Delta},\\
                
				\frac{\d^{2}}{\d\rho^{2}_{J}}\mathcal{H}_{t}\left(\rho_{J}\right)=-\frac{t}{4\Omega(E\left(\rho_{J}\right))^{3}}\left(\frac{\d E}{\d\rho_{J}}\right)^{2}+\frac{t}{2\Omega(E\left(\rho_{J}\right))}\frac{\d^{2}E}{\d\rho^{2}_{J}},\\
                
				\frac{\d^{3}}{\d\rho^{3}_{J}}\mathcal{H}_{t}\left(\rho_{J}\right)=\frac{3t}{8\Omega(E\left(\rho_{J}\right))^{5}}\left(\frac{\d E}{\d\rho_{J}}\right)^{3}-\frac{3t}{4\Omega\left(E\left(\rho_{J}\right)\right)^{3}}\frac{\d E}{\d\rho_{J}}\frac{\d^{2}E}{\d\rho^{2}_{J}}+\frac{t}{2\Omega\left(E\left(\rho_{J}\right)\right)}\frac{\d^{3}E}{\d\rho^{3}_{J}}.\\
			\end{cases}
		\end{equation}
		Moreover, using the definition of $E=E\left(\rho_{J}\right)$ we have
		\begin{equation}\label{eq:lambda's derivatives}
			\begin{cases}
				\frac{\d}{\d\rho_{J}}E\left(\rho_{J}\right)=\frac{1}{\rho^{\prime}_{J}},\\
				\frac{\d^{2}}{\d\rho^{2}_{J}}E\left(\rho_{J}\right)=-\frac{\rho_{J}^{\prime\prime}}{\left(\rho^{\prime}_{J}\right)^{3}},\\
				\frac{\d^{3}}{\d\rho^{3}_{J}}E\left(\rho_{J}\right)=\frac{3\left(\rho^{\prime\prime}\right)^{2}}{\left(\rho^{\prime}_{J}\right)^{5}}-\frac{\rho_{J}^{\prime\prime\prime}}{\left(\rho^{\prime}_{J}\right)^{4}}.\\
			\end{cases}
		\end{equation}
		We begin with assuming $\mathcal{L}_{\kappa}\subset\Gamma_{0}^{\left(J\right)}$. In this case, noting that 
		\begin{equation*}
			A_{0}=
			\begin{pmatrix}
				-E&-1\\
				1&0
			\end{pmatrix}
		\end{equation*}
		and by the fact that 
		\begin{equation*}
			\left|A_{J}-A_{0}\right|_{\mathcal{C}^{3}\left(\Gamma_{0}^{\left(J\right)}\right)}\leq\varepsilon_{0}^{\frac{1}{2}},
		\end{equation*}
		we deduce
		\begin{equation*}
			\left|\left({\rm tr}A_{J}\right)^{\prime}+1\right|_{\Gamma_{0}^{\left(J\right)}},\;\left|\left({\rm tr}A_{J}\right)^{\prime\prime}\right|,\;\left|\left({\rm tr}A_{J}\right)^{\prime\prime\prime}\right|\leq\varepsilon_{0}^{\frac{1}{2}}.
		\end{equation*}
		On $\Gamma_{0}^{(J)}$, by the definition of $\rho_{J}$, we have $\rho_{J}=\xi_{J}$. Applying the identity $\xi^{\prime}_{J}=-\frac{\left({\rm tr}A_{J}\right)^{\prime}}{2\sin\xi_{J}}$ on $\Gamma_{0}^{(J)}$ and \eqref{eq:Ht's derivatives}, \eqref{eq:lambda's derivatives}, we have
		\begin{equation*}
			\begin{cases}
				\frac{\d^{2}}{\d\rho^{2}_{J}}\mathcal{H}_{t}\left(\rho_{J}\right)=-\frac{t}{\Omega\left(E\left(\rho_{J}\right)\right)^{3}}\left(\frac{\sin\rho_{J}}{\left({\rm tr}A_{J}\right)^{\prime}}\right)^{2}-\frac{t}{\Omega\left(E\left(\rho_{J}\right)\right)}\left(\frac{2\left({\rm tr}A_{J}\right)^{\prime\prime}\sin^{2}\rho_{J}}{\left(\left({\rm tr}A_{J}\right)^{\prime}\right)^{3}}+\frac{\cos\rho_{J}}{\left({\rm tr}A_{J}\right)^{\prime}}\right),\\
				\frac{\d^{3}}{\d\rho^{3}_{J}}\mathcal{H}_{t}\left(\rho_{J}\right)=-\frac{3t}{\Omega\left(E\left(\rho_{J}\right)\right)^{5}}\left(\frac{\sin\rho_{J}}{\left({\rm tr}A_{J}\right)^{\prime}}\right)^{3}-\frac{3t}{\Omega\left(E\left(\rho_{J}\right)\right)^{3}}\frac{\sin\rho_{J}}{\left({\rm tr}A_{J}\right)^{\prime}}\left(\frac{2\left({\rm tr}A_{J}\right)^{\prime\prime}\sin^{2}\rho_{J}}{\left(\left({\rm tr}A_{J}\right)^{\prime}\right)^{3}}+\frac{\cos\rho_{J}}{\left({\rm tr}A_{J}\right)^{\prime}}\right)\\
				\qquad\qquad\qquad+\frac{t}{\Omega\left(E\left(\rho_{J}\right)\right)}\left(\frac{\sin\rho_{J}}{\left({\rm tr}A_{J}\right)^{\prime}}+\frac{4\left({\rm tr}A_{J}\right)^{\prime\prime\prime}\sin^{3}\rho_{J}}{\left(\left({\rm tr}A_{J}\right)^{\prime}\right)^{4}}-\frac{12\left(\left({\rm tr}A_{J}\right)^{\prime\prime}\right)^{2}\sin^{3}\rho_{J}}{\left(\left({\rm tr}A_{J}\right)^{\prime}\right)^{5}}-\frac{6\left({\rm tr}A_{J}\right)^{\prime\prime}\cos\rho_{J}\sin\rho_{J}}{\left(\left({\rm tr}A_{J}\right)^{\prime}\right)^{3}}\right).\\
			\end{cases}
		\end{equation*}
    Using these identities we have
    \begin{equation*}
        \begin{aligned}
            &\left|\frac{\d^{2}}{\d\rho^{2}_{J}}\mathcal{H}_{t}\left(\rho_{J}\right)-\left(-\frac{t}{\Omega\left(E\left(\rho_{J}\right)\right)^{3}}\sin^{2}\rho_{J}+\frac{t}{\Omega\left(E\left(\rho_{J}\right)\right)}\cos\rho_{J}\right)\right|\\
            \leq&\frac{t}{\Omega\left(E\left(\rho_{J}\right)\right)^{3}}\sin^{2}\rho_{J}\left|1-\frac{1}{\left(\left({\rm tr}A_{J}\right)^{\prime}\right)^{2}}\right|+\frac{t}{\Omega\left(E\left(\rho_{J}\right)\right)}\left|\cos\rho_{J}\right|\left|1+\frac{1}{\left({\rm tr}A_{J}\right)^{\prime}}\right|+\\
            &\quad \frac{t}{\Omega\left(E\left(\rho_{J}\right)\right)}\cdot\left|\frac{2\left({\rm tr}A_{J}\right)^{\prime\prime}\sin^{2}\rho_{J}}{\left(\left({\rm tr}A_{J}\right)^{\prime}\right)^{3}}\right|\\
            \leq&\frac{t}{C\left(m\right)}\left(12\varepsilon_{0}^{\frac{1}{2}}+2\varepsilon_{0}^{\frac{1}{2}}+16\varepsilon_{0}^{\frac{1}{2}}\right)=\frac{30t\varepsilon_{0}^{\frac{1}{2}}}{C\left(m\right)},\\
        \end{aligned}
    \end{equation*}
    where $C\left(m\right):=\min\left\{\inf\Omega\left(E\left(\rho_{J}\right)\right)^{3},\inf\Omega\left(E\left(\rho_{J}\right)\right)\right\}$ and
    \begin{equation*}
        \begin{aligned}
            &\left|\frac{\d^{3}}{\d\rho^{3}_{J}}\mathcal{H}_{t}\left(\rho_{J}\right)-\left(\frac{3t}{\Omega\left(E\left(\rho_{J}\right)\right)^{5}}\sin^{3}\rho_{J}-\frac{3t}{\Omega\left(E\left(\rho_{J}\right)\right)^{3}}\sin\rho_{J}\cos\rho_{J}-\frac{t}{\Omega\left(E\left(\rho_{J}\right)\right)}\sin\rho_{J}\right)\right|\\
            \leq&\frac{3t}{\Omega\left(E\left(\rho_{J}\right)\right)^{5}}\left|\sin^{3}\rho_{J}\right|\left|1+\frac{1}{\left(\left({\rm tr}A_{J}\right)^{\prime}\right)^{3}}\right|+\frac{3t}{\Omega\left(E\left(\rho_{J}\right)\right)^{3}}\left|\sin\rho_{J}\cos\rho_{J}\right|\left|1-\frac{1}{\left(\left({\rm tr}A_{J}\right)^{\prime}\right)^{2}}\right|+\\
            &\quad \frac{t}{\Omega\left(E\left(\rho_{J}\right)\right)}\left|\sin\rho_{J}\right|\left|1+\frac{1}{\left({\rm tr}A_{J}\right)^{\prime}}\right|+\frac{6t}{\Omega\left(E\left(\rho_{J}\right)\right)^{3}}\left|\frac{\left({\rm tr}A_{J}\right)^{\prime\prime}\sin^{3}\rho_{J}}{\left(\left({\rm tr}A_{J}\right)^{\prime}\right)^{4}}\right|+\\
            &\quad \frac{t}{\Omega\left(E\left(\rho_{J}\right)\right)}\left(\left|\frac{4\left({\rm tr}A_{J}\right)^{\prime\prime\prime}\sin^{3}\rho_{J}}{\left(\left({\rm tr}A_{J}\right)^{\prime}\right)^{4}}\right|+\left|\frac{12\left(\left({\rm tr}A_{J}\right)^{\prime\prime}\right)^{2}\sin^{3}\rho_{J}}{\left(\left({\rm tr}A_{J}\right)^{\prime}\right)^{5}}\right|+\left|\frac{6\left({\rm tr}A_{J}\right)^{\prime\prime}\cos\rho_{J}\sin\rho_{J}}{\left(\left({\rm tr}A_{J}\right)^{\prime}\right)^{3}}\right|\right)\\
            \leq&\frac{t}{C\left(m\right)}\left(168\varepsilon^{\frac{1}{2}}_{0}+36\varepsilon_{0}^{\frac{1}{2}}+2\varepsilon^{\frac{1}{2}}_{0}+96\varepsilon^{\frac{1}{2}}_{0}+64\varepsilon^{\frac{1}{2}}_{0}+384\varepsilon^{\frac{1}{2}}_{0}+48\varepsilon^{\frac{1}{2}}_{0}\right)=\frac{798t\varepsilon_{0}^{\frac{1}{2}}}{C\left(m\right)},\\
        \end{aligned}
    \end{equation*}
    where $C\left(m\right):=\min\left\{\inf\Omega\left(E\left(\rho_{J}\right)\right),\inf\Omega\left(E\left(\rho_{J}\right)\right)^{3},\inf\Omega\left(E\left(\rho_{J}\right)\right)^{5}\right\}$. Noting that if the following equation holds
    \begin{equation*}
        \frac{t}{\Omega\left(E\left(\rho_{J}\right)\right)^{3}}\sin^{2}\rho_{J}=\frac{t}{\Omega\left(E\left(\rho_{J}\right)\right)}\cos\rho_{J},
    \end{equation*}
    we deduce that
    \begin{equation*}
        \left|\frac{3t}{\Omega\left(E\left(\rho_{J}\right)\right)^{5}}\sin^{3}\rho_{J}-\frac{3t}{\Omega\left(E\left(\rho_{J}\right)\right)^{3}}\sin\rho_{J}\cos\rho_{J}-\frac{t}{\Omega\left(E\left(\rho_{J}\right)\right)}\sin\rho_{J}\right|>0.
    \end{equation*}
    This implies that for any $\rho_{J}\in\left(-\varepsilon_{J}^{\frac{1}{4}},\pi+\varepsilon_{J}^{\frac{1}{4}}\right)$, we always have
    \begin{equation*}
        \left|\frac{\d^{2}}{\d\rho^{2}_{J}}\mathcal{H}_{t}\left(\rho_{J}\right)\right|+\left|\frac{\d^{3}}{\d\rho^{3}_{J}}\mathcal{H}_{t}\left(\rho_{J}\right)\right|\geq C_{0}\left(m\right)t,
    \end{equation*}
    for some $C_{0}\left(m\right)>0$ which depends on $m$ for all sufficiently small $\varepsilon_{0}>0$. 
    
    Recall that on $\Gamma_{0}^{\left(J\right)}$, 
    \begin{equation*}
        A_{0}=
        \begin{pmatrix}
            -E&-1\\
            1&0\\ 
        \end{pmatrix},
    \end{equation*}
    hence we have ${\rm tr}\left(A_{0}\left(E\right)\right)=-E$, $\left({\rm tr}\left(A_{0}\left(E\right)\right)\right)^{\prime}=-1$, $\left({\rm tr}\left(A_{0}\left(E\right)\right)\right)^{\prime\prime}=\left({\rm tr}\left(A_{0}\left(E\right)\right)\right)^{\prime\prime\prime}=0$. Combining the fact 
    \begin{equation*}
        \left|{\rm tr}\left(A_{J}\left(E\right)\right)-{\rm tr}\left(A_{0}\left(E\right)\right)\right|_{\mathcal{C}^{3}(\Gamma_{0}^{\left(J\right)})}\leq\varepsilon_{0}^{\frac{1}{2}},
    \end{equation*}
    and $\rho_{J}=\xi_{J}=\arccos\left(\frac{1}{2}{\rm tr}A_{J}\right)$ on $\Gamma_{0}^{\left(J\right)}$, we derive the following result 
    \begin{equation}\label{eq: DKG-error estimate}
        \left|2\cos\rho_{J}+E\left(\rho_{J}\right)\right|_{\mathcal{C}^{3}(\rho_{J}(\Gamma_{0}^{\left(J\right)}))}\leq\varepsilon_{0}^{\frac{1}{2}}.
    \end{equation}

    As for the function
\begin{equation*}
    \widetilde{\Omega}\left(\rho\right)=\sqrt{2+m^{2}-2\cos\rho}\ ,
\end{equation*}
we calculate its third derivative
\begin{equation*}
    \begin{aligned}
        \left|\widetilde{\Omega}^{\prime\prime\prime}\left(\rho\right)\right|=&\left|-\frac{\sin\rho}{\widetilde{\Omega}\left(\rho\right)}
    -\frac{3\cos\rho\sin\rho}{\widetilde{\Omega}\left(\rho\right)^{3}}
    +\frac{3\sin^{3}\rho}{\widetilde{\Omega}\left(\rho\right)^{5}}\right|\\
    =&\frac{\left|\sin\rho\right|}{\widetilde{\Omega}\left(\rho\right)^{5}}
    \left|\left(m^{2}+2\right)^{2}-\left(m^{2}+2\right)\cos\rho+\cos^{2}\rho-3\right|.
    \end{aligned}
\end{equation*}
    Since $m^{2}>0$, we obtain that there exist constants $C_{1}\left(m\right),C_{2}\left(m\right)>0$ depending on $m$, such that
    \begin{equation}\label{eq: Free DKG estimate-3rd derivative}
        C_{1}\left(m\right)\left|\sin\rho\right|\leq\left|\widetilde{\Omega}^{\prime\prime\prime}\left(\rho\right)\right|\leq C_{2}\left(m\right)\left|\sin\rho\right|.
    \end{equation}
    It follows from \eqref{eq: DKG-error estimate} and \eqref{eq: Free DKG estimate-3rd derivative} that
    \begin{equation*}
        \left|\frac{1}{t}\frac{\d^{3}}{\d\rho^{3}_{J}}\mathcal{H}_{t}\left(\rho_{J}\right)-\frac{\d^{3}}{\d\rho^{3}_{J}}\widetilde{\Omega}\left(\rho_{J}\right)\right|\leq C_{3}\left(m\right)\cdot\varepsilon_{0}^{\frac{1}{2}},
    \end{equation*}
    and 
    \begin{equation}
        \frac{C_{1}\left(m\right)}{2}\left|\sin\rho_{J}\right|\leq\left|\frac{1}{t}\frac{\d^{3}}{\d\rho^{3}_{J}}\mathcal{H}_{t}\left(\rho_{J}\right)\right|_{}\leq2C_{2}\left(m\right)\left|\sin\rho_{J}\right|,
    \end{equation}
    on $\rho_{J}(\Gamma_{0}^{\left(J\right)})$.

    We divide the interval $\rho_{J}\mathcal{L}_{\kappa}$ into the following two parts
    \begin{equation*}
        \left(\rho_{J}\mathcal{L}_{\kappa}\right)_{1}:=\rho_{J}\mathcal{L}_{\kappa}\cap\left\{\left|\sin\rho_{J}\right|\geq\frac{C_{0}\left(m\right)}{4C_{2}\left(m\right)}\right\}\quad{\rm and}\quad\left(\rho_{J}\mathcal{L}_{\kappa}\right)_{2}:=\rho_{J}\mathcal{L}_{\kappa}\cap\left\{\left|\sin\rho_{J}\right|<\frac{C_{0}\left(m\right)}{4C_{2}\left(m\right)}\right\}.
    \end{equation*}
    Then $\left(\rho_{J}\mathcal{L}_{\kappa}\right)_{1}$ consists of no more than one connected interval and $\left(\rho_{J}\mathcal{L}_{\kappa}\right)_{2}$ is the union of at most two connected intervals. On $\left(\rho_{J}\mathcal{L}_{\kappa}\right)_{1}$, we have
    \begin{equation*}
        \left|\frac{\d^{3}}{\d\rho^{3}_{J}}\mathcal{H}_{t}\left(\rho_{J}\right)\right|\geq\frac{C_{1}\left(m\right)C_{0}\left(m\right)}{8C_{2}\left(m\right)}\cdot t.
    \end{equation*}
    On a connected component of $\left(\rho_{J}\mathcal{L}_{\kappa}\right)_{2}$, we have
    \begin{equation*}
        \left|\frac{\d^{2}}{\d\rho^{2}_{J}}\mathcal{H}_{t}\left(\rho_{J}\right)\right|\geq C_{0}\left(m\right)t-2C_{2}\left(m\right)\frac{C_{0}\left(m\right)}{4C_{2}\left(m\right)}t=\frac{C_{0}\left(m\right)}{2}t.
    \end{equation*}
    Using Van der Corput's Lemma (see, for instance, \cite{Stein}), we derive
    \begin{equation*}
        \begin{aligned}
            &\left|\int_{\rho_{J}\left(\mathcal{L}_{\kappa}\cap S_{2}\right)}\beta^J_{k,k_{\Delta}}\left(E\left(\rho_{J}\right)\right)\beta^J_{n,n_{\Delta}}\left(E\left(\rho_{J}\right)\right)e^{-{\rm i}t\big(\Omega\left(E\left(\rho_{J}\right)\right)+(k_{\Delta}-n_{\Delta})\rho_{J}\Big)}\d\rho_{J}\right|\\
            \leq&2\left(108+10\cdot 108\right)\Bigg(\left(5\cdot2^{3-1}-2\right)\cdot\left(\frac{C_{1}\left(m\right)C_{0}\left(m\right)}{8C_{2}\left(m\right)}\cdot t\right)^{-\frac{1}{3}}\\
            &\quad\quad\quad\quad\quad\quad\quad\quad\quad\quad\quad\quad\quad\quad\quad+2\left(5\cdot2^{2-1}-2\right)\cdot\left(\frac{C_{0}\left(m\right)}{2}\cdot t\right)^{-\frac{1}{2}}\Bigg)\\
            \leq&C_{4}\left(m\right)\JP{t}^{-\frac{1}{3}}\quad{\rm for}\;t\geq 1.
    \end{aligned}
\end{equation*}
Now we assume that $\mathcal{L}_{\kappa}\subset\Gamma^{\left(J\right)}_{j}$ for $1\leq j\leq J$. Note that in this case, we have
\begin{equation*}
    \frac{\varepsilon_{j-1}^{\frac{\sigma}{4}}}{3\,\varepsilon_{J}^{\frac{1}{20}}}\leq\frac{|({\rm tr}A_{J})'|}{2|\sin\xi_{J}|}\leq\left|\rho^{\prime}_{J}\right|=\left|\xi^{\prime}_{J}\right|\leq\frac{N_{j}^{10\eta}}{\sin\xi_{J}},\quad\left|\rho^{\prime\prime}_{J}\right|=\left|\xi^{\prime\prime}_{J}\right|\geq\frac{\varepsilon_{j}^{\frac{3\sigma}{4}}}{4\left|\sin\xi_{J}\right|^{3}}.
\end{equation*}
Set $\Omega_{-}:=\inf_{E\in\Sigma}\Omega(E)>0$ and $\Omega_{+}:=\sup_{E\in\Sigma}\Omega(E)<\infty$. Hence on $\mathcal{L}_{\kappa}$, we deduce

\begin{equation}\label{eq:Htpp-lower-j}
\begin{aligned}
    \left|\frac{\d^{2}}{\d\rho_{J}^{2}}\mathcal{H}_{t}\left(\rho_{J}\right)\right|
    &\geq \frac{t}{2\Omega_{+}}
    \left|\frac{\d^{2}E}{\d\rho_{J}^{2}}\right|
    -\frac{t}{4\Omega_{-}^{3}}\left|\frac{\d E}{\d\rho_{J}}\right|^{2}\\
    &\geq \frac{t}{8\Omega_{+}}\frac{\varepsilon_{j-1}^{\frac{3\sigma}{4}}}{N_{j-1}^{30\eta}}
    -\frac{9t}{4\Omega_{-}^{3}}\varepsilon_{J}^{\frac{1}{10}}\varepsilon_{j-1}^{-\frac{\sigma}{2}}.
\end{aligned}
\end{equation}
Since $J\geq j$, we have $\varepsilon_{J}\leq \varepsilon_{j-1}^{1+\sigma}$, hence
$\varepsilon_{J}^{\frac{1}{10}}\varepsilon_{j-1}^{-\frac{\sigma}{2}}
\leq \varepsilon_{j-1}^{\frac{1+\sigma}{10}-\frac{\sigma}{2}}$.
In particular, by taking $\varepsilon_{0}$ sufficiently small (depending only on $m,\gamma,\eta,d,r$),
the second term in \eqref{eq:Htpp-lower-j} can be absorbed by the first one, and we obtain the uniform lower bound
\begin{equation}\label{eq:Htpp-final-j}
    \left|\frac{\d^{2}}{\d\rho_{J}^{2}}\mathcal{H}_{t}\left(\rho_{J}\right)\right|
    \geq C_{5}\left(m\right)t\frac{\varepsilon_{j-1}^{\frac{3\sigma}{4}}}{N_{j-1}^{30\eta}}\geq C_{5}\left(m\right)t\varepsilon_{j}^{\frac{13\sigma}{16}}\qquad {\rm on}\;\mathcal{L}_{\kappa},
\end{equation}

Using Van der Corput's Lemma, we have
\begin{equation*}
    \begin{aligned}
        &\left|\int_{\rho_{J}\left(\mathcal{L}_{\kappa}\cap S_{2}\right)}\beta^J_{k,k_{\Delta}}\left(E\left(\rho_{J}\right)\right)\beta^J_{n,n_{\Delta}}\left(E\left(\rho_{J}\right)\right)e^{-{\rm i}t\big(\Omega\left(E\left(\rho_{J}\right)\right)+(k_{\Delta}-n_{\Delta})\rho_{J}\Big)}\d\rho_{J}\right|\\
        \leq&2\left(3\varepsilon_{j}^{3\sigma}+10\cdot 3\varepsilon_{j}^{3\sigma}\right)\left(5\cdot2^{2-1}-2\right)\cdot\left(C_{5}\left(m\right)t\varepsilon_{j}^{\frac{13\sigma}{16}}\right)^{-\frac{1}{2}}\\
        \leq&C_{6}\left(m\right)\varepsilon_{j}^{2\sigma}\JP{t}^{-\frac{1}{2}}\quad{\rm for}\;t\geq 1.
    \end{aligned}
\end{equation*}
By the following fact
\begin{equation*}
    \#\left\{\kappa:\mathcal{L}_{\kappa}\right\}\leq 2\left|\ln \varepsilon_{0}\right|^{2J^{2}d},
\end{equation*}
we obtain that
\begin{equation*}
    \left|\mathbb{I}^{J}(S_{2};k_{\Delta},n_{\Delta},\rho_{J})\right|=\left|\sum_{\kappa}\mathbb{I}^{J}(\mathcal{L}_{\kappa}\cap S_{2};k_{\Delta},n_{\Delta},\rho_{J})\right|\leq C_{7}\left(m\right)\left|\ln\varepsilon_{0}\right|^{2J^{2}d}\JP{t}^{-\frac{1}{3}}.
\end{equation*}
This completes the proof.
\end{proof}

\begin{lemma}\label{lem:small}
    For any $J\in\Z_{\geq 1}$ and $t\geq 1$, we have
    \begin{equation*}
        \left|\mathbb{I}^{J}\left(\Sigma\right)\right|\leq61\varepsilon_{J}^{\frac{3\sigma}{4}}+100\left|k_{\Delta}-n_{\Delta}\right|\varepsilon^{\frac{1}{10}}_{J}+54|\ln\varepsilon_{0}|^{2J^{2}d}\,\varepsilon_{J}^{\frac{1}{20}}+C\left(m\right)\left|\ln\varepsilon_{0}\right|^{2J^{2}d}\JP{t}^{-\frac{1}{3}}.
    \end{equation*}
\end{lemma}

\begin{proof}
    This Lemma follows directly from Lemmas \ref{lem:small+error}, \ref{lem:S1} and \ref{lem:S2}.
\end{proof}

\begin{proof}[Proof of Lemma \ref{lem:disper-log}]

		Combining Lemma \ref{lem:large} and Lemma \ref{lem:small}, we deduce
		\begin{equation*}
			\begin{aligned}
				&\left|\mathbb{I}\left(\Sigma\right)\right|\\
				\leq&32\varepsilon_{J}^{\frac{3\sigma}{4}}+\min\Bigg\{\frac{18}{\left|k_{\Delta}-n_{\Delta}\right|}\left|\ln \varepsilon_{0}\right|^{2J^{2}d}+\frac{3\Theta_{1}\JP{t}}{\left|k_{\Delta}-n_{\Delta}\right|}\left(\left(1+\varepsilon_{0}^{\frac{1}{4}}\right)^{2}+\varepsilon_{J}^{6\sigma}\right)\left(1+\frac{1}{\Theta_{2}}\right),\\
				&\quad\qquad\quad\qquad\quad61\varepsilon_{J}^{\frac{3\sigma}{4}}+100\left|k_{\Delta}-n_{\Delta}\right|\varepsilon^{\frac{1}{10}}_{J}+54|\ln\varepsilon_{0}|^{2J^{2}d}\,\varepsilon_{J}^{\frac{1}{20}}+C\left(m\right)\left|\ln\varepsilon_{0}\right|^{2J^{2}d}\JP{t}^{-\frac{1}{3}}\Bigg\}\\
				\leq&C\left(m\right)\left(\varepsilon_{J}^{\frac{3\sigma}{4}}+\min\left\{\frac{\left|\ln \varepsilon_{0}\right|^{2J^{2}d}}{\left|k_{\Delta}-n_{\Delta}\right|}+\frac{\JP{t}}{\left|k_{\Delta}-n_{\Delta}\right|},\varepsilon_{J}^{\frac{3\sigma}{4}}+\varepsilon^{\frac{1}{10}}_{J}\left|k_{\Delta}-n_{\Delta}\right|+|\ln\varepsilon_{0}|^{2J^{2}d}\left(\varepsilon_{J}^{\frac{1}{20}}+\JP{t}^{-\frac{1}{3}}\right)\right\}\right)
			\end{aligned}
		\end{equation*}
		for any $J\geq 1$. Let $t\geq 1$ be a fixed time, if $\left|k_{\Delta}-n_{\Delta}\right|\geq\JP{t}^{\frac{4}{3}}$, then we have
		\begin{equation*}
			\left|\mathbb{I}\left(\Sigma\right)\right|\leq C\left(m\right)\left(\varepsilon_{J}^{\frac{3\sigma}{4}}+\frac{\left|\ln \varepsilon_{0}\right|^{2J^{2}d}}{\JP{t}^{\frac{4}{3}}}+\frac{1}{\JP{t}^{\frac{1}{3}}}\right)
		\end{equation*}
		where $J\geq 1$ will be chosen later. If $\left|k_{\Delta}-n_{\Delta}\right|\leq\JP{t}^{\frac{4}{3}}$, we choose the smallest $J$ such that
		\begin{equation*}
			\varepsilon_{J}^{\frac{3\sigma}{4}}\leq\JP{t}^{-\frac{5}{3}}.
		\end{equation*}
		Hence we have
		\begin{equation*}
			J=\lceil \frac{1}{\ln\left(1+\sigma\right)}\ln\left(\frac{20}{9\sigma}\frac{\ln\JP{t}}{\left|\ln\varepsilon_{0}\right|}\right)\rceil
		\end{equation*}
		and for sufficiently small $\varepsilon_{0}$, 
		\begin{equation*}
			\left|\ln\varepsilon_{0}\right|^{2J^{2}d}\leq\left|\ln\varepsilon_{0}\right|^{2d\left(1+\frac{1}{\ln\left(1+\sigma\right)}\ln\ln\left(2+\JP{t}\right)\right)^{2}}\leq\left|\ln\varepsilon_{0}\right|^{82000d\left(\ln\ln\left(2+\JP{t}\right)\right)^{2}}.
		\end{equation*}
		So using this $J$ we derive
		\begin{equation*}
			\left|\mathbb{I}\left(\Sigma\right)\right|\leq C\left(m\right)\left|\ln\varepsilon_{0}\right|^{82000d\left(\ln\ln\left(2+\JP{t}\right)\right)^{2}}\JP{t}^{-\frac{1}{3}}.
		\end{equation*}
		This completes the proof.

\end{proof}

\section{Applications}\label{sec:4}

Throughout this section we assume the hypotheses of Theorem~\ref{thm:disper}.
Fix $\theta\in\T^d$ and write
\[
        G:=G_\theta,\qquad
        \mathbf A:=\sqrt{G+m^2}.
\]
Since $m>0$ and $\varepsilon_0=|V|_r$ is sufficiently small, we may assume
\[
        G+m^2\ge c_m I
\]
on $\ell^2(\Z)$ for some constant $c_m>0$ independent of $\theta$. In particular,
$\mathbf A$ is a bounded, positive, self-adjoint operator on $\ell^2(\Z)$ and
$\mathbf A^{-1}$ is bounded on $\ell^2(\Z)$.

\subsection{Auxiliary boundedness of $\mathbf A^{-1}$}

We shall use the following standard consequence of the Combes--Thomas estimate, see, for example \cite{CT73}.

\begin{lemma}\label{lem:Ainv-lp}
For every $1\le p\le \infty$, the operator $\mathbf A^{-1}$ is bounded on
$\ell^p(\Z)$. More precisely,
\[
        \|\mathbf A^{-1}f\|_{\ell^p(\Z)}
        \le C_{m,p}\|f\|_{\ell^p(\Z)},
        \qquad 1\le p\le\infty,
\]
where the constant is independent of $\theta$.
\end{lemma}

\begin{proof}
Set $T:=G+m^2$. By the smallness of $\varepsilon_0$, $T\ge c_m I$.
The standard Combes--Thomas estimate for one-dimensional discrete
Schr\"odinger operators gives, uniformly in $\theta$, $\lambda\ge0$, and
$n,k\in\Z$,
\[
        \left|
        \left\langle \delta_n,(T+\lambda)^{-1}\delta_k\right\rangle
        \right|
        \le
        C_m(1+\lambda)^{-1}e^{-c_m(1+\lambda)^{-1/2}|n-k|}.
\]
Using the Balakrishnan representation \cite{Bal60}
\[
        \mathbf A^{-1}=T^{-1/2}
        =
        \frac1\pi\int_0^\infty
        \lambda^{-1/2}(T+\lambda)^{-1}\,\d\lambda,
\]
we obtain a kernel bound
\[
        \left|
        \left\langle \delta_n,\mathbf A^{-1}\delta_k\right\rangle
        \right|
        \le C_m e^{-c_m'|n-k|}.
\]
The Schur test then implies the asserted $\ell^p$ boundedness for all
$1\le p\le\infty$.
\end{proof}

\subsection{Strichartz estimates}

We use the half-wave dispersive estimate obtained in Section~\ref{sec:3}:
for every $0<\tau<1/3$,
\begin{equation}\label{eq:half-wave-dispersive-app}
        \|e^{\pm {\rm i}t\mathbf A}f\|_{\ell^\infty(\Z)}
        \le C_\tau \JP{t}^{-\tau}\|f\|_{\ell^1(\Z)},
        \qquad t\in\R.
\end{equation}
Together with the unitarity of $e^{\pm{\rm i}t\mathbf A}$ on $\ell^2(\Z)$,
the abstract Keel--Tao theorem \cite[Theorem 1.2]{KT98} gives the following half-wave Strichartz
estimates: if $(q,r)$ and $(\tilde q,\tilde r)$ are $\tau$--admissible pairs
with $q,\tilde q<\infty$, then
\begin{equation}\label{eq:half-wave-Str-hom}
        \|e^{\pm {\rm i}t\mathbf A}f\|_{L^q_t(\R;\ell^r(\Z))}
        \le C\|f\|_{\ell^2(\Z)}
\end{equation}
and
\begin{equation}\label{eq:half-wave-Str-inhom}
        \left\|
        \int_0^t e^{\pm{\rm i}(t-s)\mathbf A}F(s)\,\d s
        \right\|_{L^q_t(\R;\ell^r(\Z))}
        \le
        C\|F\|_{L^{\tilde q'}_t(\R;\ell^{\tilde r'}(\Z))}.
\end{equation}
The constants are independent of $\theta$.

\begin{proof}[Proof of Theorem~\ref{thm:Strichartz}]
The Duhamel formula for \eqref{eq: inhomogegeneous} is
\[
u(t)
=
\cos(t\mathbf A)\varphi
+\mathbf A^{-1}\sin(t\mathbf A)\psi
+
\int_0^t
\mathbf A^{-1}\sin((t-s)\mathbf A)F(s)\,\d s .
\]
Using
\[
        \cos(t\mathbf A)
        =\frac12\left(e^{{\rm i}t\mathbf A}+e^{-{\rm i}t\mathbf A}\right),
        \qquad
        \mathbf A^{-1}\sin(t\mathbf A)
        =
        \frac{1}{2{\rm i}}
        \left(e^{{\rm i}t\mathbf A}-e^{-{\rm i}t\mathbf A}\right)\mathbf A^{-1},
\]
we decompose $u$ into two homogeneous half-wave terms and two inhomogeneous
half-wave terms. By \eqref{eq:half-wave-Str-hom} and the $\ell^2$ boundedness
of $\mathbf A^{-1}$,
\[
        \|\cos(t\mathbf A)\varphi\|_{L^q_t\ell^r}
        +
        \|\mathbf A^{-1}\sin(t\mathbf A)\psi\|_{L^q_t\ell^r}
        \le
        C\bigl(\|\varphi\|_{\ell^2}+\|\psi\|_{\ell^2}\bigr).
\]
For the Duhamel term, \eqref{eq:half-wave-Str-inhom} and
Lemma~\ref{lem:Ainv-lp} with $p=\tilde r'$ give
\[
\begin{aligned}
        \left\|
        \int_0^t
        \mathbf A^{-1}\sin((t-s)\mathbf A)F(s)\,\d s
        \right\|_{L^q_t\ell^r}
        &\le
        C\|\mathbf A^{-1}F\|_{L^{\tilde q'}_t\ell^{\tilde r'}_n}  \\
        &\le
        C\|F\|_{L^{\tilde q'}_t\ell^{\tilde r'}_n}.
\end{aligned}
\]
Combining the preceding estimates proves \eqref{eq:inhom-Strichartz}.
\end{proof}

\subsection{Small-data global well-posedness}

\begin{proof}[Proof of Theorem~\ref{thm:nonlinear}]
Let
\[
        N(u):=|u|^{p-1}u,
        \qquad
        \tau:=\frac{2}{p-1}.
\]
Since $p>7$, we have $0<\tau<1/3$, and the pair
\[
        (q,r)=(p+1,p+1)
\]
is $\tau$--admissible. Set
\[
        S:=L^{p+1}_t(\R;\ell^{p+1}(\Z)),
        \qquad
        \|u\|_S:=\|u\|_{L^{p+1}_t\ell^{p+1}_n}.
\]
The nonlinear equation can be written as
\[
        \partial_{tt}u+\mathbf A^2u=\pm N(u),
\]
and its Duhamel formula is
\begin{equation}\label{eq:Duhamel-NLKG}
        u(t)
        =
        \cos(t\mathbf A)\varphi
        +
        \mathbf A^{-1}\sin(t\mathbf A)\psi
        \pm
        \int_0^t
        \mathbf A^{-1}\sin((t-s)\mathbf A)N(u(s))\,\d s .
\end{equation}

By Theorem~\ref{thm:Strichartz} with $(q,r)=(\tilde q,\tilde r)=(p+1,p+1)$,
\[
        \|\Phi(u)\|_S
        \le
        C\bigl(\|\varphi\|_{\ell^2}+\|\psi\|_{\ell^2}\bigr)
        +
        C\|N(u)\|_{L^{\frac{p+1}{p}}_t
        \ell^{\frac{p+1}{p}}_n}.
\]
Since
\[
        \|N(u)\|_{L^{\frac{p+1}{p}}_t
        \ell^{\frac{p+1}{p}}_n}
        =
        \|u\|_S^p,
\]
and
\[
        |N(a)-N(b)|
        \le
        C_p\bigl(|a|^{p-1}+|b|^{p-1}\bigr)|a-b|,
\]
we also have
\[
        \|\Phi(u)-\Phi(v)\|_S
        \le
        C\bigl(\|u\|_S^{p-1}+\|v\|_S^{p-1}\bigr)\|u-v\|_S.
\]
Hence, if
\[
        \|\varphi\|_{\ell^2}+\|\psi\|_{\ell^2}\le \varepsilon
\]
with $\varepsilon>0$ sufficiently small, $\Phi$ is a contraction on the ball
\[
        \{u\in S:\|u\|_S\le 2C\varepsilon\}.
\]
This gives a unique global solution satisfying
\begin{equation}\label{eq:S-bound}
        \|u\|_S\le C\varepsilon .
\end{equation}

The solution also satisfies the energy bound
\begin{equation}\label{eq:L2-energy-bound}
        \sup_{t\in\R}
        \left(
        \|u(t)\|_{\ell^2}
        +
        \|\partial_tu(t)\|_{\ell^2}
        \right)
        \le C\varepsilon .
\end{equation}
Indeed, in the defocusing case this follows from conservation of the standard
energy. In the focusing case, the same identity combined with the embedding
$\ell^2(\Z)\hookrightarrow\ell^{p+1}(\Z)$ and the smallness of the data
absorbs the potential nonlinear contribution.

Applying Theorem~\ref{thm:Strichartz} to \eqref{eq:Duhamel-NLKG} and using
\eqref{eq:S-bound} gives, for every $\tau$--admissible pair $(q,r)$ with
$q<\infty$,
\[
        \|u\|_{L^q_t(\R;\ell^r(\Z))}
        \le
        C\varepsilon
        +
        C\|u\|_S^p
        \le
        C\varepsilon .
\]
The endpoint pair $(q,r)=(\infty,2)$ follows from \eqref{eq:L2-energy-bound}.
Thus the asserted bounds hold for all admissible pairs.

It remains to prove the decay of fixed-time $\ell^r$ norms. Let $r>2$ and let
$(q_r,r)$ be the corresponding $\tau$--admissible pair. Then
$u\in L^{q_r}_t(\R;\ell^r)$ with $q_r<\infty$. Moreover, by
\eqref{eq:L2-energy-bound}, the function
\[
        f_r(t):=\|u(t)\|_{\ell^r}
\]
is uniformly continuous on $\R$, since
\[
        |f_r(t)-f_r(s)|
        \le
        \|u(t)-u(s)\|_{\ell^r}
        \le
        \|u(t)-u(s)\|_{\ell^2}
        \le
        C\varepsilon |t-s|.
\]
A uniformly continuous function belonging to $L^{q_r}(0,\infty)$ must converge
to zero as $t\to+\infty$. Therefore
\[
        \|u(\cdot,t)\|_{\ell^r(\Z)}\to0,
        \qquad t\to+\infty,
        \quad r>2.
\]
The proof for $t\to-\infty$ is identical.
\end{proof}

\subsection{Scattering}

We next prove Theorem~\ref{thm:scatter}. It is convenient to use the half-wave
variable
\begin{equation}\label{eq:def-w-scattering}
        w(t):=u(t)+{\rm i}\mathbf A^{-1}\partial_tu(t).
\end{equation}
Then
\[
        u(t)=\RE w(t),
        \qquad
        \partial_tu(t)=\mathbf A\,\IM w(t),
\]
and the nonlinear equation is equivalent to
\begin{equation}\label{eq:w-duhamel}
        w(t)
        =
        e^{-{\rm i}t\mathbf A}w(0)
        \pm
        {\rm i}\int_0^t
        e^{-{\rm i}(t-s)\mathbf A}\mathbf A^{-1}N(u(s))\,\d s .
\end{equation}

\begin{lemma}\label{lem:small-data-scattering}
Let $u$ be the global small-data solution obtained in
Theorem~\ref{thm:nonlinear}. Then there exists a unique
$w_+\in\ell^2(\Z)$ such that
\begin{equation}\label{eq:halfwave-scatter}
        \lim_{t\to+\infty}
        \|w(t)-e^{-{\rm i}t\mathbf A}w_+\|_{\ell^2}=0.
\end{equation}
An analogous statement holds as $t\to-\infty$.
\end{lemma}

\begin{proof}
Multiplying \eqref{eq:w-duhamel} by $e^{{\rm i}t\mathbf A}$ gives
\[
        e^{{\rm i}t\mathbf A}w(t)
        =
        w(0)
        \pm
        {\rm i}\int_0^t
        e^{{\rm i}s\mathbf A}\mathbf A^{-1}N(u(s))\,\d s .
\]
For $0<t_1<t_2$, by unitarity of $e^{{\rm i}t\mathbf A}$ on $\ell^2$ and
Lemma~\ref{lem:Ainv-lp},
\[
\begin{aligned}
        \left\|
        e^{{\rm i}t_2\mathbf A}w(t_2)
        -
        e^{{\rm i}t_1\mathbf A}w(t_1)
        \right\|_{\ell^2}
        &\le
        C
        \|N(u)\|_{L^{\frac{p+1}{p}}_t((t_1,t_2);
        \ell^{\frac{p+1}{p}}_n)}  \\
        &=
        C
        \|u\|_{L^{p+1}_t((t_1,t_2);\ell^{p+1}_n)}^p.
\end{aligned}
\]
Since $u\in S$, the right-hand side tends to zero as
$t_1,t_2\to+\infty$. Thus $e^{{\rm i}t\mathbf A}w(t)$ is Cauchy in
$\ell^2$, and hence converges to some $w_+\in\ell^2$. This proves
\eqref{eq:halfwave-scatter}. Uniqueness follows from the unitarity of the
linear group. The proof as $t\to-\infty$ is identical.
\end{proof}

\begin{proof}[Proof of Theorem~\ref{thm:scatter}]
We prove the statement as $t\to+\infty$; the negative-time case is the same.

First, Lemma~\ref{lem:small-data-scattering} gives scattering for every
sufficiently small initial datum. To obtain the wave operator, fix a small
asymptotic state $(\varphi_+,\psi_+)\in\ell^2\times\ell^2$ and set
\[
        w_+:=\varphi_+ +{\rm i}\mathbf A^{-1}\psi_+ .
\]
We seek $w_0\in\ell^2$ such that the corresponding solution satisfies
\[
        \lim_{t\to+\infty}
        \|w(t)-e^{-{\rm i}t\mathbf A}w_+\|_{\ell^2}=0.
\]
Equivalently, using \eqref{eq:w-duhamel}, $w_0$ must solve
\begin{equation}\label{eq:wave-operator-fixed-point}
        w_0
        =
        w_+
        \mp
        {\rm i}\int_0^\infty
        e^{{\rm i}s\mathbf A}\mathbf A^{-1}N(u_{w_0}(s))\,\d s ,
\end{equation}
where $u_{w_0}$ denotes the solution generated by the initial half-wave datum
$w_0$.

Define
\[
        \mathcal T(w_0)
        :=
        w_+
        \mp
        {\rm i}\int_0^\infty
        e^{{\rm i}s\mathbf A}\mathbf A^{-1}N(u_{w_0}(s))\,\d s .
\]
The integral is convergent in $\ell^2$ by the same estimate used in
Lemma~\ref{lem:small-data-scattering}. Moreover, Theorem~\ref{thm:nonlinear}
and the Lipschitz estimate
\[
        \|N(u_1)-N(u_2)\|_{L^{\frac{p+1}{p}}_t\ell^{\frac{p+1}{p}}_n}
        \le
        C\bigl(\|u_1\|_S^{p-1}+\|u_2\|_S^{p-1}\bigr)
        \|u_1-u_2\|_S
\]
imply
\[
        \|\mathcal T(w_0)-w_+\|_{\ell^2}
        \le C\|u_{w_0}\|_S^p
        \le C\|w_0\|_{\ell^2}^p
\]
and, for small $w_0,\widetilde w_0$,
\[
        \|\mathcal T(w_0)-\mathcal T(\widetilde w_0)\|_{\ell^2}
        \le
        C\delta^{p-1}\|w_0-\widetilde w_0\|_{\ell^2}.
\]
Choosing $\delta>0$ sufficiently small, $\mathcal T$ is a contraction on a
small ball in $\ell^2$. Therefore there exists a unique small $w_0$ solving
\eqref{eq:wave-operator-fixed-point}.

Finally define
\[
        \varphi:=\RE w_0,
        \qquad
        \psi:=\mathbf A\,\IM w_0 .
\]
Since $\mathbf A$ is bounded on $\ell^2$, we have
$(\varphi,\psi)\in\ell^2(\Z)\times\ell^2(\Z)$. The relation
$u=\RE w$ and $\partial_tu=\mathbf A\,\IM w$ converts the half-wave scattering
statement into
\[
        \lim_{t\to+\infty}
        \left\|
        \binom{u(t)}{\partial_tu(t)}
        -
        \binom{v_+(t)}{\partial_tv_+(t)}
        \right\|_{\ell^2\times\ell^2}
        =0,
\]
where $v_+$ is the linear solution with data
$(\varphi_+,\psi_+)$. This proves Theorem~\ref{thm:scatter}.
\end{proof}

\section*{Acknowledgement}
  We sincerely thank Professors Wen Huang, Shiping Liu and Qi Zhou for their discussions.  H.Z. is supported by the National Key R \& D Program of China 2023YFA1010200 and the
 National Natural Science Foundation of China No. 12431004. Z.W. is supported by the National Natural Science Foundation of China No. 12341102.

\subsection*{Availability of Data}  No data was used for the research described in the article.

\subsection*{Declarations of Conflict of Interest}
The authors do not have any possible conflicts of
interest.

	\bibliographystyle{amsplain}

\end{document}